\documentclass[a4paper,11pt,twoside]{article}
\RequirePackage[a4paper,head=1cm]{geometry}
\usepackage[utf8x]{inputenc}
\usepackage{textcomp}
\usepackage{amsmath,amsfonts,verbatim,afterpage,euscript,mathrsfs,amssymb,empheq}
\usepackage{amsthm}
\usepackage{dsfont}
\usepackage{hyperref}
\usepackage{pst-fill,pst-grad}
\usepackage{textcomp}
\usepackage{multicol}
\usepackage{amsmath}
\usepackage{amssymb}
\usepackage{hyperref}
\usepackage{dsfont}

\newtheorem{Theoreme}{Theorem}
\newtheorem{Lemme}{Lemma}[section]
\newtheorem{Corollaire}{Corollary}[section]

\newcommand{\mysection}{\setcounter{equation}{0} \section}

\def \vu{\vec{u}}
\def \vv{\vec{v}}
\def \vf{\vec{f}}
\def \vn{\vec{\nabla}}
\title{\bf Some results for a stationary Navier-Stokes equation with a rough drift in a weighted functional framework.} 
\author{Diego Chamorro\footnote{Laboratoire de Math\'ematiques et Mod\'elisation d'Evry (LaMME) - UMR 8071. Universit\'e d'Evry Val d'Essonne, 23 Boulevard de France, 91037 Evry Cedex, France. email: \textit{diego.chamorro@univ-evry.fr}}, Anca-Nicoleta Marcoci\footnote{Department of Mathematics and Computer Science. Technical University of Civil Engineering, Bucharest, Bld. Lacul Tei, no. 124, sector 2. Romania. email: \textit{anca.marcoci@utcb.ro}}, Liviu-Gabriel Marcoci\footnote{Department of Mathematics and Computer Science. Technical University of Civil Engineering, Bucharest, Bld. Lacul Tei, no. 124, sector 2. Romania.  email: \textit{liviu.marcoci@utcb.ro}}.} 
\begin{document} 
%%%%%%%%%%%%%%%%%%%%%%%%%%%%%%%%%%%%%%%%%%%%%%%%%%%
\maketitle 
%%%%%%%%%%%%%%%%%%%%%%%%%%%%%%%%%%%%%%%%%%%%%%%%%%%
%%%%%%%%%%%%%%%%%%%%%%%%%%%%%%%%%%%%%%%%%%%%%%%%%%%
\begin{scriptsize}
\abstract{ \noindent In this article, we study some classes of solutions for a stationary Navier-Stokes equation where we consider a rough drift given by a singular integral operator which does not belong to the classical Calder\'on-Zygmund family of singular integral operators. Given a small external force, we will construct solutions to this system in the framework of weighted Morrey-Sobolev spaces. The use of Morrey-based Sobolev spaces provides a more general setting than the usual Lebesgue-based Sobolev spaces, and the presence of Muckenhoupt weights will allow us to present some existence and uniqueness results from several points of view.}\\

{\footnotesize
\noindent \textbf{MSC (2020) Primary: 35A01; Secondary: 35B33}
}
\end{scriptsize}
%%%%%%%%%%%%%%%%%%%%%%%%%%%%%%%%%%%%%%%%%%%%%%%%%%%
%\tableofcontents 

%%%%%%%%%%%%%%%%%%%%%%%%%%%%%%%%%%%%%%%%%%%%%%%%%%%
%%%%%%%%%%%%%%%%%%%%%%%%%%%%%%%%%%%%%%%%%%%%%%%%%%%
%%%%%%%%%%%%%%%%%%%%%%%%%%%%%%%%%%%%%%%%%%%%%%%%%%%
\mysection{Introduction}
In this article, we study some problems related to the existence, regularity, and uniqueness of weighted solutions for the following stationary Navier-Stokes equation with a rough drift:
\begin{equation}\label{Equation_Intro}
\Delta \vu-(\mathcal{T}(\vu) \cdot \vn) \vu-\vn \pi+\vf=0,\qquad div(\vu)=0,\qquad x\in \mathbb{R}^n, n\geq 2,
\end{equation}
where $\vu:\mathbb{R}^n\longrightarrow \mathbb{R}^n$ is the velocity vector field, $\pi:\mathbb{R}^n\longrightarrow \mathbb{R} $ is the internal pressure of the fluid  and the vector $\vf:\mathbb{R}^n\longrightarrow \mathbb{R}^n$ is a given external force. Here, the quantity $\mathcal{T}(\vu)$ is a drift vector such that $\mathcal{T}(\vu)=(\mathcal{T}(u_1), \cdots, \mathcal{T}(u_n))$ where $\mathcal{T}$ is a rough singular integral operator defined by the expression 
\begin{equation}\label{Def_Operator}
\mathcal{T}(\phi)(x)=p.v.\int_{\mathbb{R}^n}\frac{\Omega_k(y/ |y|)}{|y|^n}\phi(x-y)dy,
\end{equation}
for any locally integrable function $\phi:\mathbb{R}^n\longrightarrow \mathbb{R}$, and where the kernels $\Omega_k:\mathbb{S}^{n-1}\longrightarrow \mathbb{R}$ are such that 
$\Omega_k \in L^1(\mathbb{S}^{n-1})$, $\displaystyle{\int_{\mathbb{S}^{n-1}}\Omega_k \ d\sigma=0}$ and $\Omega_k \in L^{\rho}(\mathbb{S}^{n-1})$ with $1<\rho<n$.\\ 

One important feature of this kind of operators $\mathcal{T}$ is that they fall outside the usual framework of singular integral operators of convolution type (see Section 4.4 of the book \cite{Grafakos}) and in this sense we will say that the vector $\mathcal{T}(\vu)$ is a \emph{rough} drift for the equation (\ref{Equation_Intro}). Some boundedness properties of this general family of rough singular operators were studied in \cite{Hoang1}, \cite{Hoang} and \cite{Li} (although with less general conditions over the kernel function $\Omega$ that the ones that will be used here) and in this article we will exploit some recent estimates obtained in our previous work \cite{ChMarcociMarcoci} that generalizes the previous inequalities available for this kind of operators.\\

We are particularly interested in the properties of the solutions of the problem (\ref{Equation_Intro}) that can be obtained in a weighted functional framework, where the weights considered here will belong to some Muckenhoupt class $A_p$ and, as we will see, the presence of some well-suited weights will shed some light on the properties of the solutions studied here. Note that the use of appropriate weights can be quite useful to deepen the study of the Navier-Stokes equations, see \emph{e.g.} \cite{Fernandez} or \cite{PGLR} for the use of weights in the analysis of the evolution of the Navier-Stokes problem. Some results in a weighted setting are available for the stationary system studied here, see for example \cite{Amrouche}, \cite{Duran}, \cite{Frehse}, \cite{Frehse1}, \cite{Otarola}, \cite{Schumacher} and the references therein. Note, however, that the study of the stationary Navier-Stokes system in the weighted framework seems to be, to the best of our knowledge, limited to bounded, periodic, or exterior domains. In this sense, our results are completely new since we consider the equations in the whole Euclidean space $\mathbb{R}^n$ and we add a rough singular integral operator in the drift term. \\

Let us mention that, although some weighted boundedness estimates are available for rough type operators (see in particular \cite{Watson}, which only deals with Lebesgue spaces), in the present work we will need another type of inequalities, stated in the more general setting of weighted Morrey spaces, that will allow us to consider simultaneously a weighted and an unweighted framework (see Lemma \ref{LemmaRough} below) and this mix of information will be crucial to establish our results on the properties of solutions to the equation (\ref{Equation_Intro}). \\

In order to present our results, we need to introduce some notations. Recall first that if a locally integrable function $\omega:\mathbb{R}^n\longrightarrow ]0,+\infty[$ is a weight then, for $1\leq p\leq q<+\infty$ and for a locally integrable function $\vec{\phi}:\mathbb{R}^n\longrightarrow \mathbb{R}^n$, the weighted Morrey spaces $\mathcal{M}^{p,q}(\omega)$ are given by the condition: 
\begin{equation}\label{Def_Morrey}
\|\vec{\phi}\|_{\mathcal{M}^{p,q}(\omega)}= \underset{r>0}{\sup}\underset{x\in \mathbb{R}^n}{\sup}\ \frac{1}{\omega(B(x,r))^{\frac{1}{p}-\frac{1}{q}}}\left(\int_{B(x,r)}|\vec{\phi}(y)|^p\omega(y)dy\right)^{\frac{1}{p}}<+\infty,
\end{equation}
where $\omega(B(x,r))=\displaystyle{\int_{B(x,r)}\omega(y)dy}$. Note that when $\omega\equiv 1$, we recover the usual Morrey spaces $\mathcal{M}^{p,q}$. Remark also that if $p=q$ then we recover the usual (weighted) Lebesgue spaces as we have the identification $L^q(\omega)=\mathcal{M}^{q,q}(\omega)$. Associated to these weighted Morrey spaces, for $s\in\mathbb{R}$, we can define the weighted Sobolev spaces $\dot{W}^{s}_{p,q}(\omega)$ by the condition 
\begin{equation}\label{Def_SobolevMorrey}
\|\vec{\phi}\|_{\dot{W}^{s}_{p,q}(\omega)}=\|(-\Delta)^{\frac{s}{2}}\vec{\phi}\|_{\mathcal{M}^{p,q}(\omega)}<+\infty,
\end{equation}
where the fractional power of the Laplacian $(-\Delta)^{\frac{s}{2}}$ can be defined in the Fourier level by its symbol $|\xi|^{s}$. As before, in the particular case when $\omega\equiv 1$, we will simply write 
$\dot{W}^{s}_{p,q}$ to denote the unweighted Sobolev space associated to the Morrey space $\mathcal{M}^{p,q}$. See more properties of these functional spaces in the Section \ref{Secc_Notations} below.\\

Our first result gives the global framework that we are going to study here: 
%%%%%%%%%%%%%%%%%%%%%%%%%%%%%%%%%%%%%%%%%%%%%%%%%%%
\begin{Theoreme}[\bf Existence in mixed spaces]\label{Theorem_Existence}
Over the Euclidean space $\mathbb{R}^n$ with $n\geq 3$, consider a Muckenhoupt weight $\omega$ that belongs to the class $A_{\frac{p}{2}}$ with $2<p< n$. Let $\vf:\mathbb{R}^n\longrightarrow \mathbb{R}^n$ be an external force such that $\vf \in \dot{W}^{-1}_{\frac{p}{2}, \frac{n}{2}}\cap \dot{W}^{-1}_{\frac{p}{2}, \frac{n}{2}}(\omega)$. If the quantity $\|\vf\|_{\dot{W}^{-1}_{\frac{p}{2}, \frac{n}{2}}}+\|\vf\|_{\dot{W}^{-1}_{\frac{p}{2}, \frac{n}{2}}(\omega)}$ is small enough, then there exists a solution $(\vu, \pi)$ of the stationary Navier-Stokes equations \eqref{Equation_Intro} such that we have $\vu\in \dot{W}^1_{\frac{p}{2},\frac{n}{2}}\cap \dot{W}^1_{\frac{p}{2},\frac{n}{2}}(\omega)$ and $\pi\in \mathcal{M}^{\frac{p}{2},\frac{n}{2}}\cap \mathcal{M}^{\frac{p}{2},\frac{n}{2}}(\omega)$.  
\end{Theoreme}
%%%%%%%%%%%%%%%%%%%%%%%%%%%%%%%%%%%%%%%%%%%%%%%%%%%
Some remarks are in order here. First note that when $p=n$, we recover a weighted Lebesgue space framework, and if moreover we set $\omega=1$, we obtain an existence result in a classical Lebesgue/Sobolev setting (see the Corollary \ref{CoroSimple} below). However, even in this simplified setting, the previous theorem seems to be new -to the best of our knowledge- due to the presence of the rough singular integral operator $\mathcal{T}$ in the equation (\ref{Equation_Intro}). Second, in the case when we consider a fully weighted setting, the situation is slightly more delicate as we do not dispose of well-suited inequalities for any generic Muckenhoupt weight and we will need to impose a special behavior to the weights in order to perform our computations (see Theorem \ref{Theo_FullyWeighted} below). In this sense, the previous result given in the Theorem \ref{Theorem_Existence} above presents an interesting compromise between an unweighted setting and a fully weighted one that will require some special conditions over the weights. \\

Note also that the previous existence theorem requires a ``smallness'' condition for the external force $\vf$ and the solution $(\vu, \pi)$ is, by construction, also ``small'' in the sense that the norm of $(\vu, \pi)$ in the space $\dot{W}^1_{\frac{p}{2},\frac{n}{2}}\cap \dot{W}^1_{\frac{p}{2},\frac{n}{2}}(\omega)\times \mathcal{M}^{\frac{p}{2},\frac{n}{2}}\cap \mathcal{M}^{\frac{p}{2},\frac{n}{2}}(\omega)$ is small (see the general Lemma \ref{LemmeBP} below for the details).\\

Next, we remark that the choice of the indexes $2<p< n$ (that define the Morrey spaces and the Sobolev spaces used here as well as the class $A_p$ for the weight $\omega$), is a consequence of the functional inequalities that allow us to control the operator $\mathcal{T}$ and in the case of the previous result, we only dispose of some estimates stated in the Lemma \ref{LemmaRough} below. \\

Note now that the structure of the functional spaces studied above is the same for the weighted spaces and for the unweighted ones, since we are using as ``base'' spaces the Morrey spaces. These spaces appear in a very natural manner when dealing with inequalities that involve rough singular operators (see \cite{ChMarcociMarcoci} and \cite{ChMarcociMarcoci1}). Finally, remark that in the particular setting considered in the Theorem \ref{Theorem_Existence}, which mixes unweighted and weighted norms, the study of solutions in another functional framework seems to be a completely open problem that will probably require new functional inequalities involving the operator $\mathcal{T}$ and the properties of the weights considered.\\

We study now a fully weighted version of the previous result, but for this we need to introduce some notation and we will say that a weight $\omega$ satisfies the ${\bf d}$-lower Ahlfors condition if we have 
\begin{equation}\label{Ahlfors}
Cr^{\bf d}\leq \omega(B(x,r)),\quad \mbox{for all } r>0, \  x\in \mathbb{R}^n.
\end{equation}
Note that this is an additional condition imposed on a Muckenhoupt weight, and a generic Muckenhoupt weight does not need to satisfy it. See \cite{Bjorn}, \cite{Hoang}, \cite{Samko} for more details and some examples on this particular hypothesis.\\ 

With the help of this ${\bf d}$-lower Ahlfors condition we can state the following result: 
%%%%%%%%%%%%%%%%%%%%%%%%%%%%%%%%%%%%%%%%%%%%%%%%%%%
\begin{Theoreme}[\bf Fully Weighted existence]\label{Theo_FullyWeighted}
Over the Euclidean space $\mathbb{R}^n$ with $n\geq 3$, consider a Muckenhoupt weight $\omega \in A_{\frac{p}{2}}$ that satisfies the ${\bf d}$-lower Ahlfors condition (\ref{Ahlfors}). If $\vf:\mathbb{R}^n\longrightarrow \mathbb{R}^n$ is an external force such that the quantity $\|\vf\|_{\dot{W}^{-1}_{\frac{p}{2}, \frac{\bf d}{2}}(\omega)}$ is small enough (with $2<p\leq {\bf d}$), then there exists one solution $(\vu, \pi)$ of the stationary Navier-Stokes equations \eqref{Equation_Intro} such that we have $\vu\in \dot{W}^1_{\frac{p}{2},\frac{\bf d}{2}}(\omega)$ and $\pi\in \mathcal{M}^{\frac{p}{2},\frac{\bf d}{2}}(\omega)$.
\end{Theoreme}
%%%%%%%%%%%%%%%%%%%%%%%%%%%%%%%%%%%%%%%%%%%%%%%%%%%
\noindent Remark that due to the lower Ahlfors condition (\ref{Ahlfors}), it is possible to consider different spaces of resolution: indeed, instead of considering Morrey-Sobolev spaces with indexes $(\frac{p}{2}, \frac{n}{2})$ with $2<p< n$ where the dimension $n$ is completely fixed, we can work with some parameters $(\frac{p}{2}, \frac{\bf d}{2})$ with $2<p\leq {\bf d}$, where the index ${\bf d}$ only depends on the weight. Note that the condition $2<p$ is necessary since we need to use the boundedness of the Riesz transforms and the set of parameters $(\frac{p}{2}, \frac{\bf d}{2})$ are needed since we are dealing with a bilinear term and at some point we will use the H\"older inequality in order to close the fixed point argument. \\

In the next result, we give a criterion for the uniqueness of the solutions obtained in the previous theorem:
%%%%%%%%%%%%%%%%%%%%%%%%%%%%%%%%%%%%%%%%%%%%%%%%%%%
\begin{Theoreme}[\bf Uniqueness]\label{Theorem_Uniqueness}
Consider $\vu, \vv\in \dot{W}^1_{\frac{p}{2},\frac{n}{2}}\cap \dot{W}^1_{\frac{p}{2},\frac{n}{2}}(\omega)$ be two solutions associated to the same external force $\vf\in \dot{W}^{-1}_{\frac{p}{2}, \frac{n}{2}}\cap \dot{W}^{-1}_{\frac{p}{2}, \frac{n}{2}}(\omega)$ of the equation (\ref{Equation_Intro}) where $\omega\in A_{\frac{p}{2}}$ with $2<p< n$.\\

\noindent If $\|\vu\|_{\dot{W}^{1}_{\frac{p}{2}, \frac{n}{2}}}+\|\vv\|_{\dot{W}^{1}_{\frac{p}{2}, \frac{n}{2}}}$ is small enough, then we have $\vu=\vv$.
\end{Theoreme}
%%%%%%%%%%%%%%%%%%%%%%%%%%%%%%%%%%%%%%%%%%%%%%%%%%%
\noindent Let us remark that only a smallness information in the unweighted norm $\dot{W}^{1}_{\frac{p}{2}, \frac{n}{2}}$ is required to deduce the uniqueness of the solutions. This fact allows us, when dealing with uniqueness problems, to dissociate the information available on the unweighted norm from that available on the weighted norm. Let us mention also that the use of the (simpler) unweighted space $\dot{W}^{1}_{\frac{p}{2}, \frac{n}{2}}$ is related to the fact that in this case the functional inequalities are easier to apply. A symmetrical uniqueness result in the weighted norm $\dot{W}^{1}_{\frac{p}{2}, \frac{n}{2}}(\omega)$ is possible but it requires, as mentioned before, some additional conditions over the weights. See the details in the Appendix \ref{AppendixUniqueness}. \\

Let us show now how the previous uniqueness result can be applied in a more general framework. Indeed, consider $\vu \in \dot{W}^1_{\frac{p}{2},\frac{n}{2}}\cap \dot{W}^1_{\frac{p}{2},\frac{n}{2}}(\omega)$ be a solution of the equation (\ref{Equation_Intro}) obtained by the Theorem \ref{Theorem_Existence} above and which is associated to one (small) external force $\vf\in \dot{W}^{-1}_{\frac{p}{2}, \frac{n}{2}}\cap \dot{W}^{-1}_{\frac{p}{2}, \frac{n}{2}}(\omega)$. Note that, by construction, this solution $\vu$ is \emph{small} in the sense that we have $\|\vu\|_{ \dot{W}^1_{\frac{p}{2},\frac{n}{2}}}+\|\vu\|_{ \dot{W}^1_{\frac{p}{2},\frac{n}{2}}(\omega)}\ll\epsilon$. Let now $\vv$ be another solution of the equation (\ref{Equation_Intro}) associated to the same small external force $\vf$, but such that $\|\vv\|_{\dot{W}^1_{\frac{p}{2},\frac{n}{2}}}\ll\epsilon$ and $\|\vv\|_{ \dot{W}^1_{\frac{p}{2},\frac{n}{2}}(\omega)}\gg 1$, \emph{i.e.} the norm of $\vv$ in the space $\dot{W}^1_{\frac{p}{2},\frac{n}{2}}$ is small while the norm of $\vv$ in the space $\dot{W}^1_{\frac{p}{2},\frac{n}{2}}(\omega)$ is big: remark that this particular situation can be obtained without any problem by using an adapted weight $\omega$. Note moreover that this solution $\vv$ \emph{cannot} be obtained following the Theorem \ref{Theorem_Existence} since its global norm in the space $\dot{W}^1_{\frac{p}{2},\frac{n}{2}}\cap \dot{W}^1_{\frac{p}{2},\frac{n}{2}}(\omega)$ is not small. Nevertheless, and despite this fact, Theorem \ref{Theorem_Uniqueness} allows us to study if the solutions $\vu$ and $\vv$ are different or not and since we have $\|\vu\|_{\dot{W}^1_{\frac{p}{2},\frac{n}{2}}}+\|\vv\|_{\dot{W}^1_{\frac{p}{2},\frac{n}{2}}}\ll\epsilon$, we can actually deduce that $\vu\equiv \vv$. As we can see, the separated study given in the Theorem \ref{Theorem_Uniqueness} allows us to obtain a uniqueness result even outside the framework of the Theorem \ref{Theorem_Existence} and we can make explicit the use of different weights that will allow to separate the information on the mixed norms considered above. \\

The plan of the article is the following. In Section \ref{Secc_Notations} we recall some properties of the Morrey spaces as well as the essential inequalities that will be used in the sequel. In Section \ref{Secc_Theo1} we prove the Theorem \ref{Theorem_Existence}, the Theorem \ref{Theo_FullyWeighted} is studied in Section \ref{SeccTheo_FullyWeighted}, while Section \ref{Secc_Theo2} is devoted to the Theorem \ref{Theorem_Uniqueness}. The Appendix \ref{Secc_Weighted_ineq} is devoted to some inequalities that will be necessary, and in the Appendix \ref{AppendixUniqueness}, we study some uniqueness results in a fully weighted framework. 

%%%%%%%%%%%%%%%%%%%%%%%%%%%%%%%%%%%%%%%%%%%%%%%%%%%
\mysection{Weights, functional spaces and operators}\label{Secc_Notations}
We first recall that a measurable function $\omega:\mathbb{R}^n\longrightarrow ]0,+\infty[$ is a Muckenhoupt weight in the class $A_p$ for some parameter $p>1$ if we have the condition
$$[\omega]_{A_{p}}=\underset{B}{\sup}\left(\frac{1}{|B|}\int_{B}\omega(x)dx\right)\left(\frac{1}{|B|}\int_{B}\omega(x)^{-\frac{1}{p-1}}dx\right)^{p-1}<+\infty.$$
Recall also that for $1<p_0<p_1$, we have the class inclusion $A_{p_0}\subset A_{p_1}$ (see the book \cite{Grafakos} for more general details about Muckenhoupt weights). 
The fact that a weight $\omega$ belongs to a Muckenhoupt class $A_p$ allows us to obtain a number of interesting properties. Indeed, for some indexes $1<p\leq q<+\infty$, if $\omega\in A_p$ we have the following estimates: 
$$\|\mathscr{M}(\vec{\phi})\|_{\mathcal{M}^{p,q}(\omega)}\leq C\|\vec{\phi}\|_{\mathcal{M}^{p,q}(\omega)},$$
where $\mathscr{M}(\vec{\phi})=\underset{x\in B}{\sup}\ \frac{1}{|B|}\displaystyle{\int_{B}|\vec{\phi}(y)|dy}$ is the classical Hardy-Littlewood maximal function and we also have 
$$\|R_j(\vec{\phi})\|_{\mathcal{M}^{p,q}(\omega)}\leq C\|\vec{\phi}\|_{\mathcal{M}^{p,q}(\omega)},$$
where $R_j$ are the Riesz transforms. See \cite[Theorem 3.2 \& Theorem 3.3]{Komori} for a proof of these two estimates. \\

Let us also recall that we have the following version of the H\"older inequality in the setting of the weighted Morrey spaces:
\begin{equation}\label{Holder}
\|\vec{\phi}\cdot \vec{\psi}\|_{\mathcal{M}^{p,q}(\omega)}\leq \|\vec{\phi}\|_{\mathcal{M}^{p_0,q_0}(\omega)}\|  \vec{\psi}\|_{\mathcal{M}^{p_1,q_1}(\omega)},
\end{equation}
for $1\leq p, p_0, p_1, q, q_0, q_1\leq +\infty$ such that $p\leq q$, $p_0\leq q_0$, $p_1\leq q_1$ and
 $\frac{1}{p}=\frac{1}{p_0}+\frac{1}{p_1}$,  $\frac{1}{q}=\frac{1}{q_0}+\frac{1}{q_1}$. 
We will need the following identity valid for any indexes $1\leq p\leq q<+\infty$ and $s>0$ such that $1\leq sp\leq sq$:
\begin{equation}\label{PuissanceFracMorrey}
\||\vec{\phi}|^s\|_{\mathcal{M}^{p,q}(\omega)}=\|\vec{\phi}\|_{\mathcal{M}^{sp,sq}(\omega)}^s.
\end{equation} 
Now, for $s,\alpha\in \mathbb{R}$ and $1<p\leq q<+\infty$, we have the identities
$$\|(-\Delta)^{\frac{\alpha}{2}}\vec{\phi}\|_{\dot{W}^{s}_{p,q}(\omega)}=\|(-\Delta)^{\frac{\alpha}{2}}(-\Delta)^{\frac{s}{2}}\vec{\phi}\|_{\mathcal{M}^{p,q}(\omega)}=\|\vec{\phi}\|_{\dot{W}^{s+\alpha}_{p,q}(\omega)},$$
and we also have the boundedness of the maximal functions (see \cite{Komori}) and of the Riesz transforms in the weighted Morrey-Sobolev spaces (see \cite{Komori}): 
$$\|\mathscr{M}(\vec{\phi})\|_{\dot{W}^{s}_{p,q}(\omega)}\leq C\|\vec{\phi}\|_{\dot{W}^{s}_{p,q}(\omega)} \qquad \mbox{and}\qquad \|R_j(\vec{\phi})\|_{\dot{W}^{s}_{p,q}(\omega)}\leq C\|\vec{\phi}\|_{\dot{W}^{s}_{p,q}(\omega)}.$$
We need now to present the main tools that will be used here, \emph{i.e.} the functional inequalities that will allow us to perform the computations. We can distinguish three cases: the ``unweighted'' inequalities in the setting of the usual Morrey spaces, the ``mixed'' estimates that involve unweighted Morrey spaces \emph{and} weighted ones and, finally, the ``fully'' weighted controls where we only consider weighted Morrey spaces. Although some of the unweighted inequalities that we need here are available in the literature, it is worth to mention here that the mixed and the fully weighted estimates 
have only been developed very recently. 
%%%%%%%%%%%%%%%%%%%%%%%%%%%%%%%%%%%%%%%%%%%%%%%%%%%
\begin{Lemme}[\bf Unweighted inequalities]\label{LemmeNoWeight}
Over the space $\mathbb{R}^n$ with $n\geq 2$, consider a vector field $\vec{\phi}:\mathbb{R}^n\longrightarrow \mathbb{R}$. If $\vec{\phi}\in \dot{W}^{\alpha}_{p,q}$ with $0<\alpha<n$, $1<p\leq q<+\infty$ and $\alpha q<n$, then we have the Sobolev inequality 
$$\|\vec{\phi}\|_{\mathcal{M}^{p^*, q^*}}\leq C\|\vec{\phi}\|_{\dot{W}^{\alpha}_{p, q}},$$
where $p^*=\frac{np}{n-\alpha q}$ and $q^*=\frac{nq}{n-\alpha q}$. 
\end{Lemme}
This estimate is a consequence of the Adams-Hedberg inequality, see \cite{Adams}, \cite{Komori}, \cite{PGLR} and the references therein for more details. Now, if we want to consider a mixed information between unweighted and weighted spaces, we have: 
%%%%%%%%%%%%%%%%%%%%%%%%%%%%%%%%%%%%%%%%%%%%%%%%%%%
\begin{Lemme}[\bf Mixed inequalities]\label{LemmeMixedWeight}
Let $\vec{\phi}:\mathbb{R}^n\longrightarrow \mathbb{R}$ be a function such that $\vec{\phi}\in \dot{W}^{\alpha}_{p,q}\cap  \dot{W}^{\alpha}_{p,q}(\omega)$ with $0<\alpha<n$, $1<p\leq q<+\infty$ and $\alpha q<n$. If $\omega\in A_p$ is a Muckenhoupt weight, then we have the following mixed estimate
$$\|\vec{\phi}\|_{\mathcal{M}^{p^*, q^*}(\omega)}\leq C\|\vec{\phi}\|_{\dot{W}^{\alpha}_{p, q}(\omega)}^{1-\frac{\alpha q}{n}}\|\vec{\phi}\|_{\dot{W}^{\alpha}_{p, q}}^{\frac{\alpha q}{n}},$$
where $p^*=\frac{np}{n-\alpha q}$ and $q^*=\frac{nq}{n-\alpha q}$.
\end{Lemme}
%%%%%%%%%%%%%%%%%%%%%%%%%%%%%%%%%%%%%%%%%%%%%%%%%%%
\noindent As far as we are aware and based on our current information, it seems that this weighted/unweighted inequality was not studied before and we could not find a reference for it: we are thus going to sketch its proof in the Appendix \ref{Secc_Weighted_ineq}. Paradoxically, the previous mixed result is simpler to establish than a fully weighted estimate, where we will need the additional ${\bf d}$-lower Ahlfors condition given in (\ref{Ahlfors}). With the help of this condition, we can state the following result: 
%%%%%%%%%%%%%%%%%%%%%%%%%%%%%%%%%%%%%%%%%%%%%%%%%%%
\begin{Lemme}[\bf A fully weighted inequality]\label{LemmeFullyWeighted}
Consider a function $\vec{\phi}:\mathbb{R}^n\longrightarrow \mathbb{R}$ such that $\vec{\phi} \in \dot{W}^{\alpha}_{p,q}(\omega)$, with $0<\alpha<n$, $1<p\leq q<+\infty$, and where $\omega\in A_p$ is a Muckenhoupt weight that satisfies the ${\bf d}$-lower Ahlfors condition (\ref{Ahlfors}) with $\alpha<\frac{{\bf d}}{q}$. Then we have the estimate
$$\|\vec{\phi}\|_{\mathcal{M}^{p^*,q^*}(\omega)}\leq C\|\vec{\phi}\|_{\dot{W}^{\alpha}_{p,q}(\omega)},$$
where $p^*=\frac{{\bf d}p}{({\bf d}-\alpha q)}$ and $q^*=\frac{{\bf d}q}{({\bf d}-\alpha q)}$.
\end{Lemme}
%%%%%%%%%%%%%%%%%%%%%%%%%%%%%%%%%%%%%%%%%%%%%%%%%%%
Let us point out that the ${\bf d}$-lower Ahlfors condition is used here as a technical assumption. Whether this hypothesis can be removed or replaced by a weaker one remains open. Remark that a similar and more restrictive condition was used in \cite{Guliyev} and \cite{Gunawan} (see also the references therein) to derive analogous inequalities in a slightly different setting. Since this estimate was not treated before, we will present a proof in the appendix at the end of this article.\\

To conclude, we will need the following results on functional inequalities involving rough singular integral operators. 
%%%%%%%%%%%%%%%%%%%%%%%%%%%%%%%%%%%%%%%%%%%%%%%%%%%
\begin{Lemme}[\bf Unweighted, mixed and fully weighted Sobolev inequalities for rough operators]\label{LemmaRough} Consider $\mathcal{T}$  a rough singular integral operator of the form given in (\ref{Def_Operator}) associated to the kernel $\Omega \in L^\rho(\mathbb{S}^{n-1})$ with $1<\rho<n$. 
\begin{itemize}
\item[1)] {\bf Unweigthed inequalities}: if $\vec{\phi}:\mathbb{R}^n\longrightarrow \mathbb{R}$ is a function such that $\vec{\phi}\in \dot{W}^1_{\frac{p}{2}, \frac{n}{2}}$ with $1<\frac{\rho n}{\rho n+\rho-n}< p<n$, then we have the estimate
$$\|\mathcal{T}(\vec{\phi})\|_{\mathcal{M}^{p,n}}\leq C\|\vec{\phi}\|_{\dot{W}^1_{\frac{p}{2}, \frac{n}{2}}}.$$
\item[2)] {\bf Mixed inequalities}: if $\vec{\phi}:\mathbb{R}^n\longrightarrow \mathbb{R}$ is a function such that we have $\vec{\phi}\in \dot{W}^1_{\frac{p}{2}, \frac{n}{2}}\cap\dot{W}^1_{\frac{p}{2}, \frac{n}{2}}(\omega)$ with $1<\frac{\rho n}{\rho n+\rho-n}< p<n$ and where $\omega$ is a Muckenhoupt weight such that $\omega\in A_p$, then we have the estimate
$$\|\mathcal{T}(\vec{\phi})\|_{\mathcal{M}^{p,n}}\leq C\|\vec{\phi}\|_{\dot{W}^1_{\frac{p}{2}, \frac{n}{2}}}^{\frac{1}{2}}\|\vec{\phi}\|_{\dot{W}^1_{\frac{p}{2}, \frac{n}{2}}(\omega)}^{\frac{1}{2}}.$$
\item[3)] {\bf Fully weighted inequalities}: assume that $\vec{\phi}:\mathbb{R}^n\longrightarrow \mathbb{R}$ is a function such that  $\vec{\phi}\in \dot{W}^1_{p, q}(\omega)$ with $1<\frac{\rho n}{\rho n+\rho-n}< p<n$ and $1<p\leq q<+\infty$. Consider a Muckenhoupt weight $\omega$ that satisfies the ${\bf d}$-lower Ahlfors condition (\ref{Ahlfors}) with $1<\frac{\bf d}{q}$, then we have 
$$\|\mathcal{T}(\vec{\phi})\|_{\mathcal{M}^{p^*,q^*}(\omega)}\leq C\|\vec{\phi}\|_{\dot{W}^1_{p,q}(\omega)},$$
where $p^*=\frac{{\bf d}p}{({\bf d}-q)}$ and $q^*=\frac{{\bf d}q}{({\bf d}-q)}$.
\end{itemize}
\end{Lemme}
%%%%%%%%%%%%%%%%%%%%%%%%%%%%%%%%%%%%%%%%%%%%%%%%%%%
\noindent These inequalities are studied in a slightly more general manner (and for a wider range of parameters) in \cite{ChMarcociMarcoci} and \cite{ChMarcociMarcoci1}. 
%%%%%%%%%%%%%%%%%%%%%%%%%%%%%%%%%%%%%%%%%%%%%%%%%%%
\mysection{Proof of the Theorem \ref{Theorem_Existence}}\label{Secc_Theo1}
To begin, we recall that the Leray projector $\mathbb{P}$ is defined by the expression $\mathbb{P}(\vec{\phi})=\vec{\phi}+\vn (-\Delta)^{-1}div(\vec{\phi})$. In particular, if $div(\vu)=0$ we have $\mathbb{P}(\vu)=\vu$, moreover we have $\mathbb{P}(\vn \pi)=0$ (see \cite{Chamorro_Livre} for more properties of the Leray projector). Thus, if we apply this projector to the equation (\ref{Equation_Intro}) we obtain the expression
$$\Delta \vu=\mathbb{P}(div(\mathcal{T}(\vu)\otimes \vu))-\mathbb{P}(\vf),$$
which can be rewritten in the following manner
\begin{equation}\label{EquaPointFixe}
\vu=(-\Delta)^{-1}\mathbb{P}(\vf)-(-\Delta)^{-1}\mathbb{P}(div(\mathcal{T}(\vu)\otimes \vu)).
\end{equation}
Since the quantity $(-\Delta)^{-1}\mathbb{P}(div(\mathcal{T}(\vu)\otimes \vu))$ is bilinear, we can apply the Banach-Picard contraction principle given by the following result: 
%%%%%%%%%%%%%%%%%%%%%%%%%%%%%%%%%%%%%%%%%%%%%%%%%%%
\begin{Lemme}\label{LemmeBP}
Let $(E, \|{\cdot}\|_{E})$ be a Banach space and let $B:E\times E \longrightarrow E$ be a bilinear application such that for all $e,f\in E$ we have the control
$$\|B(e,f)\|_{E}\leq C_{B}\|e\|_{E}\|f\|_{E}.$$
If $e_{0}\in E$ is such that $\|e_{0}\|_{E}\leq \delta$ and if we have $0<\delta<\frac{1}{4 C_{B}}$,
then the equation
$$e=e_{0}-B(e,e),$$
admits a solution $e\in E$ such that $\|e\|_{E}\leq 2 \delta$.
\end{Lemme}
%%%%%%%%%%%%%%%%%%%%%%%%%%%%%%%%%%%%%%%%%%%%%%%%%%%
 See \cite[Théorème 4.1.1]{Chamorro_Livre} for a proof of this fact. Thus, in order to obtain a solution to the previous equation (\ref{EquaPointFixe}), we only need to fix a resolution space and obtain some suitable estimates.  Indeed, as announced in the Theorem \ref{Theorem_Existence}, we will consider as resolution space the weighted Sobolev space $E=\dot{W}^{1}_{\frac{p}{2}, \frac{n}{2}}\cap \dot{W}^{1}_{\frac{p}{2}, \frac{n}{2}}(\omega)$ with some index $p$ such that $2<p< n$, which will be endowed with the norm $\|\cdot\|_{E}=\|\cdot\|_{\dot{W}^{1}_{\frac{p}{2}, \frac{n}{2}}}+\|\cdot\|_{ \dot{W}^{1}_{\frac{p}{2}, \frac{n}{2}}(\omega)}$,  and we will need to prove the following inequalities:
\begin{equation}\label{EstimationForce}
\|(-\Delta)^{-1}\mathbb{P}(\vf)\|_{E}\leq \delta,
\end{equation}
and 
\begin{equation}\label{EstimationBilineaire}
\|(-\Delta)^{-1}\mathbb{P}(div(\mathcal{T}(\vu)\otimes \vu))\|_{E}\leq C\|\vu\|_{E}\|\vu\|_{E}.
\end{equation}
These two estimates will be studied separately in the lines below.\\
%%%%%%%%%%%%%%%%%%%%%%%%%%%%%%%%%%%%%%%%%%%%%%%%%%%
\begin{itemize}
\item For the external force $\vf$, we start by writing 
$$\|(-\Delta)^{-1}\mathbb{P}(\vf)\|_{E}=\underbrace{\|(-\Delta)^{-1}\mathbb{P}(\vf)\|_{\dot{W}^{1}_{\frac{p}{2}, \frac{n}{2}}}}_{(a)}+\underbrace{\|(-\Delta)^{-1}\mathbb{P}(\vf)\|_{ \dot{W}^{1}_{\frac{p}{2}, \frac{n}{2}}(\omega)}}_{(b)}.$$
For the first quantity $(a)$ above, by the definition of the Sobolev spaces given in the expression (\ref{Def_SobolevMorrey}) and by the commutation properties of the Leray projector $\mathbb{P}$ with the fractional powers of the Laplacian, we have
$$\|(-\Delta)^{-1}\mathbb{P}(\vf)\|_{\dot{W}^{1}_{\frac{p}{2}, \frac{n}{2}}}=\|(-\Delta)^{\frac{1}{2}}(-\Delta)^{-1}\mathbb{P}(\vf)\|_{\mathcal{M}^{\frac{p}{2}, \frac{n}{2}}}=\|\mathbb{P}((-\Delta)^{-\frac{1}{2}}\vf)\|_{\mathcal{M}^{\frac{p}{2}, \frac{n}{2}}}.$$
Now, using the fact that the Leray projector is bounded in the Morrey space $\mathcal{M}^{\frac{p}{2},\frac{n}{2}}$ (since we have $2<p<n$, see \cite[Proposition 6.2]{PGLR}),  we obtain
$$\|(-\Delta)^{-1}\mathbb{P}(\vf)\|_{\dot{W}^{1}_{\frac{p}{2}, \frac{n}{2}}}\leq C\|(-\Delta)^{-\frac{1}{2}}\vf\|_{\mathcal{M}^{\frac{p}{2}, \frac{n}{2}}}=C\|\vf\|_{\dot{W}^{-1}_{\frac{p}{2}, \frac{n}{2}}}<+\infty,$$
which is finite by the hypothesis over the external force.\\

For the second quantity $(b)$, by the same arguments as above we can write 
$$\|(-\Delta)^{-1}\mathbb{P}(\vf)\|_{\dot{W}^{1}_{\frac{p}{2}, \frac{n}{2}}(\omega)}=\|\mathbb{P}((-\Delta)^{\frac{1}{2}}(-\Delta)^{-1}\vf)\|_{\mathcal{M}^{\frac{p}{2}, \frac{n}{2}}(\omega)}=\|\mathbb{P}((-\Delta)^{-\frac{1}{2}}\vf)\|_{\mathcal{M}^{\frac{p}{2}, \frac{n}{2}}(\omega)}.$$
Since the Leray projector is based on the Riesz transforms (recall that $\mathbb{P}(\vec{\phi})=\vec{\phi}+\vn (-\Delta)^{-1}div(\vec{\phi})$), and since the Riesz transforms are bounded in the weighted Morrey space $\mathcal{M}^{\frac{p}{2}, \frac{n}{2}}(\omega)$ if we have that $\omega \in A_{\frac{p}{2}}$ (see \cite[Theorem 3.3]{Komori}), then we obtain the estimate
$$\|(-\Delta)^{-1}\mathbb{P}(\vf)\|_{\dot{W}^{1}_{\frac{p}{2}, \frac{n}{2}}(\omega)}\leq C\|(-\Delta)^{-\frac{1}{2}}\vf\|_{\mathcal{M}^{\frac{p}{2}, \frac{n}{2}}(\omega)}=C\|\vf\|_{\dot{W}^{-1}_{\frac{p}{2}, \frac{n}{2}}(\omega)}<+\infty.$$
With these two estimates at hand, we finally obtain the control 
\begin{equation}\label{EstimationForceExterieureMild}
\|(-\Delta)^{-1}\mathbb{P}(\vf)\|_{E}\leq C(\|\vf\|_{\dot{W}^{-1}_{\frac{p}{2}, \frac{n}{2}}}+\|\vf\|_{\dot{W}^{-1}_{\frac{p}{2}, \frac{n}{2}}(\omega)}).
\end{equation}
and this last quantity can be made small since we control the information available over the external force $\vf$.
%%%%%%%%%%%%%%%%%%%%%%%%%%%%%%%%%%%%%%%%%%%%%%%%%%%
\item For the bilinear term, by the definition of the norm $\|\cdot\|_E$, we have
\begin{eqnarray*}
\|(-\Delta)^{-1}\mathbb{P}(div(\mathcal{T}(\vu)\otimes \vu))\|_{E}&=&\|(-\Delta)^{-1}\mathbb{P}( div(\mathcal{T}(\vu)\otimes \vu))\|_{\dot{W}^{1}_{\frac{p}{2}, \frac{n}{2}}}\\
&&+\|(-\Delta)^{-1}\mathbb{P}( div(\mathcal{T}(\vu)\otimes \vu))\|_{\dot{W}^{1}_{\frac{p}{2}, \frac{n}{2}}(\omega)},
\end{eqnarray*}
and by the commutation properties of the Leray projector $\mathbb{P}$, we write
\begin{eqnarray*}
&=&\|(-\Delta)^{\frac{1}{2}}(-\Delta)^{-1}\mathbb{P}( div(\mathcal{T}(\vu)\otimes \vu))\|_{\mathcal{M}^{\frac{p}{2}, \frac{n}{2}}}+\|(-\Delta)^{\frac{1}{2}}(-\Delta)^{-1}\mathbb{P}( div(\mathcal{T}(\vu)\otimes \vu))\|_{\mathcal{M}^{\frac{p}{2}, \frac{n}{2}}(\omega)}\\
&=&\|\mathbb{P}( (-\Delta)^{\frac{1}{2}}(-\Delta)^{-1}div(\mathcal{T}(\vu)\otimes \vu))\|_{\mathcal{M}^{\frac{p}{2}, \frac{n}{2}}}+\|\mathbb{P}( (-\Delta)^{\frac{1}{2}}(-\Delta)^{-1}div(\mathcal{T}(\vu)\otimes \vu))\|_{\mathcal{M}^{\frac{p}{2}, \frac{n}{2}}(\omega)}\\
&\leq &C\|(-\Delta)^{-\frac{1}{2}}div(\mathcal{T}(\vu)\otimes \vu)\|_{\mathcal{M}^{\frac{p}{2}, \frac{n}{2}}}+C\|(-\Delta)^{-\frac{1}{2}}div(\mathcal{T}(\vu)\otimes \vu)\|_{\mathcal{M}^{\frac{p}{2}, \frac{n}{2}}(\omega)}
\end{eqnarray*}
since, by the same arguments as above, the Leray projector is bounded in these spaces. Now, recalling that the Riesz transforms are bounded in the spaces $\mathcal{M}^{\frac{p}{2}, \frac{n}{2}}$ and $\mathcal{M}^{\frac{p}{2}, \frac{n}{2}}(\omega)$, we obtain 
\begin{eqnarray*}
\|(-\Delta)^{-1}\mathbb{P}(div(\mathcal{T}(\vu)\otimes \vu))\|_{E}&\leq &C\|(-\Delta)^{-\frac{1}{2}}div(\mathcal{T}(\vu)\otimes \vu)\|_{\mathcal{M}^{\frac{p}{2}, \frac{n}{2}}}\\
&&+C\|(-\Delta)^{-\frac{1}{2}}div(\mathcal{T}(\vu)\otimes \vu)\|_{\mathcal{M}^{\frac{p}{2}, \frac{n}{2}}(\omega)}
\\
&\leq & C\|\mathcal{T}(\vu)\otimes \vu\|_{\mathcal{M}^{\frac{p}{2}, \frac{n}{2}}}+C\|\mathcal{T}(\vu)\otimes \vu\|_{\mathcal{M}^{\frac{p}{2}, \frac{n}{2}}(\omega)}.
\end{eqnarray*}
now by the H\"older inequality in the Morrey spaces with $\frac{2}{p}=\frac{1}{p}+\frac{1}{p}$ and $\frac{2}{n}=\frac{1}{n}+\frac{1}{n}$, we have:
\begin{equation}\label{EstimationPointFixeBilineaire1}
\|(-\Delta)^{-1}\mathbb{P}(div(\mathcal{T}(\vu)\otimes \vu))\|_{E}\leq \|\mathcal{T}(\vu)\|_{\mathcal{M}^{p,n}} \|\vu\|_{\mathcal{M}^{p,n}}+C\|\mathcal{T}(\vu)\|_{\mathcal{M}^{p,n}(\omega)}\|\vu\|_{\mathcal{M}^{p,n}(\omega)}.
\end{equation}
At this point we use the first point of the Lemma \ref{LemmaRough} to obtain the estimates 
$$\|\mathcal{T}(\vu)\|_{\mathcal{M}^{p,n}}\leq C\|\vu\|_{\dot{W}^1_{\frac{p}{2},\frac{n}{2}}},$$ 
as well as 
$$\|\mathcal{T}(\vu)\|_{\mathcal{M}^{p,n}(\omega)}\leq C\|(-\Delta)^{\frac{1}{2}}\vu\|_{\mathcal{M}^{\frac{p}{2},\frac{n}{2}}}^{\frac{1}{2}}\|(-\Delta)^{\frac{1}{2}}\vu\|_{\mathcal{M}^{\frac{p}{2},\frac{n}{2}}(\omega)}^{\frac{1}{2}}=C\|\vu\|_{\dot{W}^1_{\frac{p}{2},\frac{n}{2}}}^{\frac{1}{2}}\|\vu\|_{\dot{W}^1_{\frac{p}{2},\frac{n}{2}}(\omega)}^{\frac{1}{2}},$$
and if we apply the Lemma \ref{LemmeNoWeight} and the Lemma \ref{LemmeMixedWeight}, we have 
$$\|\vu\|_{\mathcal{M}^{p,n}}\leq C\|\vu\|_{\dot{W}^1_{\frac{p}{2},\frac{n}{2}}}\qquad \mbox{and}\qquad \|\vu\|_{\mathcal{M}^{p,n}(\omega)}\leq \|\vu\|_{\dot{W}^1_{\frac{p}{2},\frac{n}{2}}}^{\frac{1}{2}}\|\vu\|_{\dot{W}^1_{\frac{p}{2},\frac{n}{2}}(\omega)}^{\frac{1}{2}}.$$
With all these inequalities at hand, coming back to the expression (\ref{EstimationPointFixeBilineaire1}), we can write 
\begin{eqnarray*}
\|(-\Delta)^{-1}\mathbb{P}(div(\mathcal{T}(\vu)\otimes \vu))\|_{E}&\leq &C\|\vu\|_{\dot{W}^1_{\frac{p}{2},\frac{n}{2}}}\|\vu\|_{\dot{W}^1_{\frac{p}{2},\frac{n}{2}}}\\
&&+C\left(\|\vu\|_{\dot{W}^1_{\frac{p}{2},\frac{n}{2}}}^{\frac{1}{2}}\|\vu\|_{\dot{W}^1_{\frac{p}{2},\frac{n}{2}}(\omega)}^{\frac{1}{2}}\right)\left(\|\vu\|_{\dot{W}^1_{\frac{p}{2},\frac{n}{2}}}^{\frac{1}{2}}\|\vu\|_{\dot{W}^1_{\frac{p}{2},\frac{n}{2}}(\omega)}^{\frac{1}{2}}\right).
\end{eqnarray*}
Since $\|\vu\|_{\dot{W}^1_{\frac{p}{2},\frac{n}{2}}}\leq \|\vu\|_{\dot{W}^1_{\frac{p}{2},\frac{n}{2}}}+\|\vu\|_{\dot{W}^1_{\frac{p}{2},\frac{n}{2}}(\omega)}=\|\vu\|_{E}$ and since, by the Young inequalities for the sum we have $\|\vu\|_{\dot{W}^1_{\frac{p}{2},\frac{n}{2}}}^{\frac{1}{2}}\|\vu\|_{\dot{W}^1_{\frac{p}{2},\frac{n}{2}}(\omega)}^{\frac{1}{2}}\leq \|\vu\|_{\dot{W}^1_{\frac{p}{2},\frac{n}{2}}}+\|\vu\|_{\dot{W}^1_{\frac{p}{2},\frac{n}{2}}(\omega)}=\|\vu\|_E$, we can write 
\begin{eqnarray*}
\|(-\Delta)^{-1}\mathbb{P}(div(\mathcal{T}(\vu)\otimes \vu))\|_{E}&\leq &C\|\vu\|_{\dot{W}^1_{\frac{p}{2},\frac{n}{2}}}\|\vu\|_{\dot{W}^1_{\frac{p}{2},\frac{n}{2}}}\\
&&+C\left(\|\vu\|_{\dot{W}^1_{\frac{p}{2},\frac{n}{2}}}^{\frac{1}{2}}\|\vu\|_{\dot{W}^1_{\frac{p}{2},\frac{n}{2}}(\omega)}^{\frac{1}{2}}\right)\left(\|\vu\|_{\dot{W}^1_{\frac{p}{2},\frac{n}{2}}}^{\frac{1}{2}}\|\vu\|_{\dot{W}^1_{\frac{p}{2},\frac{n}{2}}(\omega)}^{\frac{1}{2}}\right)\\
&\leq & C\|\vu\|_E\|\vu\|_E+C\|\vu\|_E\|\vu\|_E,
\end{eqnarray*}
from which we easily deduce the expression 
\begin{equation}\label{EstimationTermeBilineaireMild}
\|(-\Delta)^{-1}\mathbb{P}(div(\mathcal{T}(\vu)\otimes \vu))\|_{E}\leq C\|\vu\|_E\|\vu\|_E,
\end{equation}
which proves the boundedness of the bilinear term in the functional space $E$.
\end{itemize}
With the estimates (\ref{EstimationForceExterieureMild}) and (\ref{EstimationTermeBilineaireMild}), to apply the Banach-Picard contraction principle we only need to assume the smallness of the external force in the sense that we should have $\|\vf\|_{\dot{W}^{-1}_{\frac{p}{2}, \frac{n}{2}}}+\|\vf\|_{\dot{W}^{-1}_{\frac{p}{2}, \frac{n}{2}}(\omega)}\leq \epsilon$, for some constant $0<\epsilon\ll 1$ small enough in order to close the fixed point argument given by the Lemma \ref{LemmeBP}: we thus obtain a function $\vu\in \dot{W}^{1}_{\frac{p}{2}, \frac{n}{2}}\cap \dot{W}^{1}_{\frac{p}{2}, \frac{n}{2}}(\omega)$ which is a solution to the equation (\ref{EquaPointFixe}).\\

To continue, we study the pressure $\pi$: due to the divergence-free property of the velocity field we can write
$$\pi=(-\Delta)^{-1}\left(div(div (\mathcal{T}(\vu)\otimes\vu))\right)-(-\Delta)^{-1}div(\vf),$$
thus, applying the $\mathcal{M}^{\frac{p}{2},\frac{n}{2}}$ norm to this expression we can write (since the Riesz transforms are bounded in the Morrey space $\mathcal{M}^{\frac{p}{2},\frac{n}{2}}$):
\begin{eqnarray*}
\|\pi\|_{\mathcal{M}^{\frac{p}{2},\frac{n}{2}}}&\leq &\|(-\Delta)^{-1}(div(div (\mathcal{T}(\vu)\otimes\vu)))\|_{\mathcal{M}^{\frac{p}{2},\frac{n}{2}}}+\|(-\Delta)^{-1}div(\vf)\|_{\mathcal{M}^{\frac{p}{2},\frac{n}{2}}}\\
&\leq & C\|\mathcal{T}(\vu)\otimes\vu\|_{\mathcal{M}^{\frac{p}{2},\frac{n}{2}}}+C\|\vf\|_{\dot{W}^{-1}_{\frac{p}{2}, \frac{n}{2}}}.
\end{eqnarray*}
By the previous computations we know that the term $\|\mathcal{T}(\vu)\otimes\vu\|_{\mathcal{M}^{\frac{p}{2},\frac{n}{2}}}$ is bounded (since it is controlled by the quantity $\|\vu\|_{\dot{W}^{1}_{\frac{p}{2},\frac{n}{2}}}$) and by hypothesis we have that $\|\vf\|_{\dot{W}^{-1}_{\frac{p}{2}, \frac{n}{2}}}<+\infty$, from which we deduce that the pressure $\pi$ belongs to the space $\mathcal{M}^{\frac{p}{2},\frac{n}{2}}$. We consider now the space $\mathcal{M}^{\frac{p}{2},\frac{n}{2}}(\omega)$ and we write 
\begin{eqnarray*}
\|\pi\|_{\mathcal{M}^{\frac{p}{2},\frac{n}{2}}(\omega)}&\leq &\|(-\Delta)^{-1}(div(div (\mathcal{T}(\vu)\otimes\vu)))\|_{\mathcal{M}^{\frac{p}{2},\frac{n}{2}}(\omega)}+\|(-\Delta)^{-1}div(\vf)\|_{\mathcal{M}^{\frac{p}{2},\frac{n}{2}}(\omega)}\\
&\leq & C\|\mathcal{T}(\vu)\otimes\vu\|_{\mathcal{M}^{\frac{p}{2},\frac{n}{2}}(\omega)}+C\|\vf\|_{\dot{W}^{-1}_{\frac{p}{2}, \frac{n}{2}}(\omega)}.
\end{eqnarray*}
As before, following the previous computations, the term $\|\mathcal{T}(\vu)\otimes\vu\|_{\mathcal{M}^{\frac{p}{2},\frac{n}{2}}(\omega)}$ can be controlled by the norms $\|\vu\|_{\dot{W}^{1}_{\frac{p}{2},\frac{n}{2}}}$ and $\|\vu\|_{\dot{W}^{1}_{\frac{p}{2},\frac{n}{2}}(\omega)}$ and the quantity $\|\vf\|_{\dot{W}^{-1}_{\frac{p}{2}, \frac{n}{2}}(\omega)}$ is bounded by hypothesis, and from this we deduce that $\pi \in \mathcal{M}^{\frac{p}{2},\frac{n}{2}}(\omega)$.\\

By combining these results, we obtain that the pressure $\pi$ belongs to the space $\mathcal{M}^{\frac{p}{2},\frac{n}{2}}\cap \mathcal{M}^{\frac{p}{2},\frac{n}{2}}(\omega)$ and this ends the proof of the Theorem \ref{Theorem_Existence}. \hfill $\blacksquare$\\
%%%%%%%%%%%%%%%%%%%%%%%%%%%%%%%%%%%%%%%%%%%%%%%%%%%

\noindent As announced in the introduction, if we set $\omega \equiv 1$, we obtain a usual (unweighted) existence result: 
%%%%%%%%%%%%%%%%%%%%%%%%%%%%%%%%%%%%%%%%%%%%%%%%%%%
\begin{Corollaire}[\bf Unweighted existence]\label{CoroSimple}
Over the Euclidean space $\mathbb{R}^n$ with $n\geq 3$, let $\vf:\mathbb{R}^n\longrightarrow \mathbb{R}^n$ be an external force such that the quantity $\|\vf\|_{\dot{W}^{-1}_{\frac{p}{2}, \frac{n}{2}}}$ is small enough, then there exists one solution $(\vu, \pi)$ of the stationary Navier-Stokes equations \eqref{Equation_Intro} such that we have $\vu\in \dot{W}^1_{\frac{p}{2},\frac{n}{2}}$ and $\pi\in \mathcal{M}^{\frac{p}{2},\frac{n}{2}}$.  
\end{Corollaire}
%%%%%%%%%%%%%%%%%%%%%%%%%%%%%%%%%%%%%%%%%%%%%%%%%%%
\noindent The proof of this result follows the same lines as above: all the inequalities needed are stated in the Lemmas \ref{LemmeNoWeight} and \ref{LemmaRough}.
%%%%%%%%%%%%%%%%%%%%%%%%%%%%%%%%%%%%%%%%%%%%%%%%%%%
\section{Proof of the Theorem \ref{Theo_FullyWeighted}}\label{SeccTheo_FullyWeighted}
The proof relies in the use of the Banach-Picard contraction principle in the weighted resolution space $E=\dot{W}^{1}_{\frac{p}{2},\frac{\bf d}{2}}(\omega)$. Indeed, from the expression 
\begin{equation}\label{SolutionWeighted}
\vu=(-\Delta)^{-1}\mathbb{P}(\vf)-(-\Delta)^{-1}\mathbb{P}(div(\mathcal{T}(\vu)\otimes \vu)),
\end{equation}
and following the ideas of the Lemma \ref{LemmeBP}, we only need to prove the boundedness of the quantity $(-\Delta)^{-1}\mathbb{P}(\vf)$ as well as the continuity of the term $(-\Delta)^{-1}\mathbb{P}(div(\mathcal{T}(\vu)\otimes \vu))$ in the space $\dot{W}^{1}_{\frac{p}{2},\frac{\bf d}{2}}(\omega)$. For the external force we write 
$$\|(-\Delta)^{-1}\mathbb{P}(\vf)\|_{\dot{W}^{1}_{\frac{p}{2},\frac{\bf d}{2}}(\omega)}=\|(-\Delta)^{\frac{1}{2}}(-\Delta)^{-1}\mathbb{P}(\vf)\|_{\mathcal{M}^{\frac{p}{2},\frac{\bf d}{2}}(\omega)}\leq \|(-\Delta)^{-\frac{1}{2}}\vf\|_{\mathcal{M}^{\frac{p}{2},\frac{\bf d}{2}}(\omega)}=\|\vf\|_{\dot{W}^{-1}_{\frac{p}{2},\frac{\bf d}{2}}(\omega)}<+\infty,$$
since the Riesz transforms and the Leray projector are bounded in the weighted Morrey space $\mathcal{M}^{\frac{p}{2},\frac{\bf d}{2}}(\omega)$ as $\omega \in A_{\frac{p}{2}}$. For the continuity of the bilinear term we write 
\begin{eqnarray*}
\|(-\Delta)^{-1}\mathbb{P}(div(\mathcal{T}(\vu)\otimes \vu))\|_{\dot{W}^{1}_{\frac{p}{2},\frac{\bf d}{2}}(\omega)}&=&\|(-\Delta)^{\frac{1}{2}}(-\Delta)^{-1}\mathbb{P}(div(\mathcal{T}(\vu)\otimes \vu))\|_{\mathcal{M}^{\frac{p}{2},\frac{\bf d}{2}}(\omega)}\\
&\leq &C \|\mathcal{T}(\vu)\otimes \vu\|_{\mathcal{M}^{\frac{p}{2},\frac{\bf d}{2}}(\omega)}\leq C\|\mathcal{T}(\vu)\|_{\mathcal{M}^{p,{\bf d}}(\omega)}\|\vu\|_{\mathcal{M}^{p,{\bf d}}(\omega)},
\end{eqnarray*}
where we used the H\"older inequality in the weighted Morrey spaces. Since the weight considered satisfies the ${\bf d}$-lower Ahlfors condition (\ref{Ahlfors}), we can apply the third point of the Lemma \ref{LemmaRough} and the Lemma \ref{LemmeFullyWeighted} (with $2<p\leq {\bf d}$, $q=\frac{\bf d}{2}$ and $p^*=p$, $q^*={\bf d}$) to obtain the estimates $\|\mathcal{T}(\vu)\|_{\mathcal{M}^{p,{\bf d}}(\omega)}\leq C\|\vu\|_{\dot{W}^1_{\frac{p}{2},\frac{\bf d}{2}}(\omega)}$ and $\|\vu\|_{\mathcal{M}^{p,{\bf d}}(\omega)}\leq C\|\vu\|_{\dot{W}^1_{\frac{p}{2},\frac{\bf d}{2}}(\omega)}$, from which we deduce the control 
$$\|(-\Delta)^{-1}\mathbb{P}(div(\mathcal{T}(\vu)\otimes \vu))\|_{\dot{W}^{1}_{\frac{p}{2},\frac{\bf d}{2}}(\omega)}\leq C\|\vu\|_{\dot{W}^1_{\frac{p}{2},\frac{\bf d}{2}}(\omega)}\|\vu\|_{\dot{W}^1_{\frac{p}{2},\frac{\bf d}{2}}(\omega)}.$$
Once we have these estimates at hand, under the smallness hypothesis for the external force, we can obtain a solution $\vu$ of the problem (\ref{SolutionWeighted}) such that $\vu\in \dot{W}^1_{\frac{p}{2},\frac{\bf d}{2}}(\omega)$. For the pressure $\pi$, we simply write 
\begin{eqnarray*}
\|\pi\|_{\mathcal{M}^{\frac{p}{2},\frac{\bf d}{2}}(\omega)}&\leq &\|(-\Delta)^{-1}\left(div(div (\mathcal{T}(\vu)\otimes\vu))\right)\|_{\mathcal{M}^{\frac{p}{2},\frac{\bf d}{2}}(\omega)}+\|(-\Delta)^{-1}div(\vf)\|_{\mathcal{M}^{\frac{p}{2},\frac{\bf d}{2}}(\omega)}\\
&\leq & \| \mathcal{T}(\vu)\otimes\vu\|_{\mathcal{M}^{\frac{p}{2},\frac{\bf d}{2}}(\omega)}+\|(-\Delta)^{-\frac{1}{2}}\vf\|_{\mathcal{M}^{\frac{p}{2},\frac{\bf d}{2}}(\omega)}\\
&\leq & C\|\vu\|_{\dot{W}^1_{\frac{p}{2},\frac{\bf d}{2}}(\omega)}\|\vu\|_{\dot{W}^1_{\frac{p}{2},\frac{\bf d}{2}}(\omega)}+C\|\vf\|_{\dot{W}^{-1}_{\frac{p}{2},\frac{\bf d}{2}}(\omega)}<+\infty,
\end{eqnarray*}
and this ends the proof of the Theorem \ref{Theo_FullyWeighted}. \hfill $\blacksquare$
%%%%%%%%%%%%%%%%%%%%%%%%%%%%%%%%%%%%%%%%%%%%%%%%%%%
\mysection{Proof of the Theorem \ref{Theorem_Uniqueness}}\label{Secc_Theo2}
Let $\vu \in \dot{W}^{1}_{\frac{p}{2},\frac{n}{2}}\cap \dot{W}^{1}_{\frac{p}{2},\frac{n}{2}}(\omega)$ and $\vv \in \dot{W}^{1}_{\frac{p}{2},\frac{n}{2}}\cap \dot{W}^{1}_{\frac{p}{2},\frac{n}{2}}(\omega)$ be two solutions of the stationary Navier-Stokes equations \eqref{Equation_Intro} and let us assume that $\vu\neq \vv$. We have, after an application of the Leray projector, the equations
$$\vu=(-\Delta)^{-1}\mathbb{P}(\vf)-(-\Delta)^{-1}\mathbb{P}(div(\mathcal{T}(\vu)\otimes \vu))\qquad\mbox{and}\qquad \vv=(-\Delta)^{-1}\mathbb{P}(\vf)-(-\Delta)^{-1}\mathbb{P}(div(\mathcal{T}(\vv)\otimes \vv)),$$
and we obtain 
$$\vu-\vv=-(-\Delta)^{-1}\mathbb{P}(div(\mathcal{T}(\vu)\otimes \vu))+(-\Delta)^{-1}\mathbb{P}(div(\mathcal{T}(\vv)\otimes \vv)),$$
which, due to bilinearity, can be rewritten as 
\begin{equation}\label{FormuleBilineaireUnicite}
\vu-\vv=-(-\Delta)^{-1}\mathbb{P}(div(\mathcal{T}(\vu-\vv)\otimes \vu))+(-\Delta)^{-1}\mathbb{P}(div(\mathcal{T}(\vv)\otimes (\vv-\vu))).
\end{equation}
Now, applying the norm $\dot{W}^{1}_{\frac{p}{2}, \frac{n}{2}}$ to the both sides of this expression, we get
\begin{eqnarray*}
\|\vu-\vv\|_{\dot{W}^{1}_{\frac{p}{2}, \frac{n}{2}}}\leq \|(-\Delta)^{-1}\mathbb{P}(div(\mathcal{T}(\vu-\vv)\otimes \vu))\|_{\dot{W}^{1}_{\frac{p}{2}, \frac{n}{2}}}+\|(-\Delta)^{-1}\mathbb{P}(div(\mathcal{T}(\vv)\otimes (\vv-\vu)))\|_{\dot{W}^{1}_{\frac{p}{2}, \frac{n}{2}}}\\
\leq \|(-\Delta)^{\frac{1}{2}}(-\Delta)^{-1}\mathbb{P}(div(\mathcal{T}(\vu-\vv)\otimes \vu))\|_{\mathcal{M}^{\frac{p}{2}, \frac{n}{2}}}+\|(-\Delta)^{\frac{1}{2}}(-\Delta)^{-1}\mathbb{P}(div(\mathcal{T}(\vv)\otimes (\vv-\vu)))\|_{\mathcal{M}^{\frac{p}{2}, \frac{n}{2}}},
\end{eqnarray*}
from which we deduce, by the properties of the Leray projector and the boundedness of the Riesz transforms in Morrey spaces, the inequality
$$\|\vu-\vv\|_{\dot{W}^{1}_{\frac{p}{2}, \frac{n}{2}}}\leq \|\mathcal{T}(\vu-\vv)\otimes \vu\|_{\mathcal{M}^{\frac{p}{2}, \frac{n}{2}}}+\|\mathcal{T}(\vv)\otimes (\vv-\vu)\|_{\mathcal{M}^{\frac{p}{2}, \frac{n}{2}}}.$$
Now by the H\"older inequality in Morrey spaces we can write
$$\|\vu-\vv\|_{\dot{W}^{1}_{\frac{p}{2}, \frac{n}{2}}}\leq \|\mathcal{T}(\vu-\vv)\|_{\mathcal{M}^{p,n}}\|\vu\|_{\mathcal{M}^{p,n}}+\|\mathcal{T}(\vv)\|_{\mathcal{M}^{p,n}} \|\vv-\vu\|_{\mathcal{M}^{p,n}}.$$
Applying the inequality $\|\mathcal{T}(\vec{\phi})\|_{\mathcal{M}^{p,n}}\leq C\|\vec{\phi}\|_{\dot{W}^{1}_{\frac{p}{2}, \frac{n}{2}}}$ (see the first point of the Lemma \ref{LemmaRough}) and the estimate $\|\vec{\phi}\|_{\mathcal{M}^{p,n}}\leq C\|\vec{\phi}\|_{\dot{W}^{1}_{\frac{p}{2}, \frac{n}{2}}}$ (see Lemma \ref{LemmeNoWeight}), we obtain the control
\begin{eqnarray*}
\|\vu-\vv\|_{\dot{W}^{1}_{\frac{p}{2}, \frac{n}{2}}}&\leq &C\|\vu-\vv\|_{\dot{W}^{1}_{\frac{p}{2}, \frac{n}{2}}}\|\vu\|_{\dot{W}^{1}_{\frac{p}{2}, \frac{n}{2}}}+C\|\vv\|_{\dot{W}^{1}_{\frac{p}{2}, \frac{n}{2}}} \|\vv-\vu\|_{\dot{W}^{1}_{\frac{p}{2}, \frac{n}{2}}}\\
&\leq & C(\|\vu\|_{\dot{W}^{1}_{\frac{p}{2}, \frac{n}{2}}}+\|\vv\|_{\dot{W}^{1}_{\frac{p}{2}, \frac{n}{2}}}) \|\vv-\vu\|_{\dot{W}^{1}_{\frac{p}{2}, \frac{n}{2}}}. 
\end{eqnarray*}
Since we assumed that $\vu\neq \vv$, we have $\|\vu-\vv\|_{\dot{W}^{1}_{\frac{p}{2}, \frac{n}{2}}}=\|\vv-\vu\|_{\dot{W}^{1}_{\frac{p}{2}, \frac{n}{2}}}\neq 0$, and we obtain the inequality 
$$1\leq  C(\|\vu\|_{\dot{W}^{1}_{\frac{p}{2}, \frac{n}{2}}}+\|\vv\|_{\dot{W}^{1}_{\frac{p}{2}, \frac{n}{2}}}).$$
Therefore if the quantity $\|\vu\|_{\dot{W}^{1}_{\frac{p}{2}, \frac{n}{2}}}+\|\vv\|_{\dot{W}^{1}_{\frac{p}{2}, \frac{n}{2}}}$ is small enough, then we obtain a contradiction, from which we deduce that 
$0=\|\vu-\vv\|_{\dot{W}^{1}_{\frac{p}{2}, \frac{n}{2}}}$, since by the Sobolev embedding given in the Lemma \ref{LemmaRough} this quantity controls the term $\|\vu-\vv\|_{\mathcal{M}^{p,n}}$, we must have $\vu=\vv$ and the proof of the Theorem \ref{Theorem_Uniqueness} is complete. \hfill $\blacksquare$
%%%%%%%%%%%%%%%%%%%%%%%%%%%%%%%%%%%%%%%%%%%%%%%%%%%
\appendix
%%%%%%%%%%%%%%%%%%%%%%%%%%%%%%%%%%%%%%%%%%%%%%%%%%%
\section{Some unweighted and weighted estimates}\label{Secc_Weighted_ineq}
Sobolev like inequalities are well known in the setting of the unweighted Morrey spaces $\mathcal{M}^{p,q}$ (see \cite{Adams} for example). However, to the best of our knowledge, these estimates in a weighted framework were not studied before and we are not aware of a reference in the existing bibliography that establishes such inequalities. 
%%%%%%%%%%%%%%%%%%%%%%%%%%%%%%%%%%%%%%%%%%%%%%%%%%%
\subsection{Proof of the Lemma \ref{LemmeMixedWeight}.}
We are going to prove the following inequality there. 
$$\|\vec{\phi}\|_{\mathcal{M}^{p^*, q^*}(\omega)}\leq C\|\vec{\phi}\|_{\dot{W}^{\alpha}_{p, q}(\omega)}^{1-\theta}\|\vec{\phi}\|_{\dot{W}^{\alpha}_{p, q}}^{\theta},$$
where $0<\alpha<n$, $1<p\leq q<+\infty$,  $p^*=\frac{np}{n-\alpha q}$ and $q^*=\frac{nq}{n-\alpha q}$.
Our starting point is the following expression
\begin{eqnarray}
|\mathcal{I}_\alpha(f)(x)|=c_\alpha\left|\int_{\mathbb{R}^n}\frac{f(x-y)}{|y|^{n-\alpha}}dy\right|\leq c_\alpha\sum_{j\in \mathbb{Z}}\int_{\{2^j<|y|\leq2^{j+1}\}}\frac{|f(x-y)|}{|y|^{n-\alpha}}dy\notag\\
\leq c_\alpha\underbrace{\sum_{j\leq \lfloor \log_2(\mathcal{K})\rfloor}\int_{\{2^j<|y|\leq2^{j+1}\}}\frac{|f(x-y)|}{|y|^{n-\alpha}}dy}_{(I_1)}+c_\alpha\underbrace{\sum_{j> \lfloor \log_2(\mathcal{K})\rfloor}\int_{\{2^j<|y|\leq2^{j+1}\}}\frac{|f(x-y)|}{|y|^{n-\alpha}}dy}_{(I_2)},\quad\label{EstimationAvantMorreyApoids}
\end{eqnarray}
where the constant $\mathcal{K}>0$ will be fixed later. We will study the two quantities above separately. 
%%%%%%%%%%%%%%%%%%%%%%%%%%%%%%%%%%%%%%%%%%%%%%%%%%%
\begin{itemize}
\item For the term $(1)$, we write
$$I_1=\sum_{j\leq \lfloor \log_2(\mathcal{K})\rfloor}\int_{\{2^j<|y|\leq2^{j+1}\}}\frac{|f(x-y)|}{|y|^{n-\alpha}}dy\leq \sum_{j\leq \lfloor \log_2(\mathcal{K})\rfloor}2^{-j(n-\alpha)}\int_{\{2^j<|y|\leq2^{j+1}\}}|f(x-y)|dy,$$
since over the set $\{2^j<|y|\leq2^{j+1}\}$ we have $\frac{1}{|y|^{n-\alpha}}\leq 2^{-j(n-\alpha)}$. We now write
\begin{eqnarray*}
I_1&\leq &\sum_{j\leq \lfloor \log_2(\mathcal{K})\rfloor}2^{-jn+ j\alpha}\frac{|B(x,2^{j+1})|}{|B(x,2^{j+1})|}\int_{B(x,2^{j+1})}|f(y)|dy\\
&\leq& v_n\sum_{j\leq \lfloor \log_2(\mathcal{K})\rfloor}\frac{2^{-jn+ j\alpha+jn+n}}{|B(x,2^{j+1})|}\int_{B(x,2^{j+1})}|f(y)|dy\\
&\leq & v_n\sum_{j\leq \lfloor \log_2(\mathcal{K})\rfloor}\frac{2^{j\alpha}}{|B(x,2^{j+1})|}\int_{B(x,2^{j+1})}|f(y)|dy,
\end{eqnarray*}
where $v_n$ is the $n$-dimensional volume of the unit ball. Now, by the definition of the Hardy-Littlewood maximal function $\mathscr{M}$ we can write
\begin{equation}\label{EstimationPointWise1}
I_1\leq C\sum_{j\leq \lfloor \log_2(\mathcal{K})\rfloor}2^{j\alpha}\mathscr{M}(f)(x)\leq C\mathcal{K}^\alpha\mathscr{M}(f)(x).
\end{equation}
%%%%%%%%%%%%%%%%%%%%%%%%%%%%%%%%%%%%%%%%%%%%%%%%%%%
\item For the term $(2)$, we have the estimates
\begin{eqnarray*}
I_2&=&\sum_{j> \lfloor \log_2(\mathcal{K})\rfloor}\int_{\{2^j<|y|\leq2^{j+1}\}}\frac{|f(x-y)|}{|y|^{n-\alpha}}dy\leq \sum_{j> \lfloor \log_2(\mathcal{K})\rfloor}2^{-j(n-\alpha)}\int_{\{2^j<|y|\leq2^{j+1}\}}|f(x-y)|dy\\
&\leq &\sum_{j> \lfloor \log_2(\mathcal{K})\rfloor}2^{-j(n-\alpha)}\int_{B(x,2^{j+1})}|f(y)|dy
\end{eqnarray*}
which we rewrite in the following manner:
\begin{eqnarray*}
I_2&\leq &\sum_{j> \lfloor \log_2(\mathcal{K})\rfloor}2^{-j(n-\alpha)}\frac{|B(x, 2^{j+1})|}{|B(x, 2^{j+1})|}\int_{B(x, 2^{j+1})}|f(y)|dy\\
&\leq &C\sum_{j> \lfloor \log_2(\mathcal{K})\rfloor}\frac{2^{j\alpha}}{|B(x, 2^{j+1})|}\int_{B(x, 2^{j+1})}|f(y)|dy.
\end{eqnarray*}
We have now
$$I_2\leq C\sum_{j> \lfloor \log_2(\mathcal{K})\rfloor}2^{j\alpha}\frac{|B(x, 2^{j+1})|^{-\frac{1}{q}}}{|B(x, 2^{j+1})|^{\frac{1}{p}-\frac{1}{q}}}\left(\int_{B(x, 2^{j+1})}|f(y)|^pdy\right)^{\frac{1}{p}},$$
which is
$$I_2\leq C\sum_{j> \lfloor \log_2(\mathcal{K})\rfloor}2^{j(\alpha-\frac{n}{q})}\frac{1}{|B(x, 2^{j+1})|^{\frac{1}{p}-\frac{1}{q}}}\left(\int_{B(x, 2^{j+1})}|f(y)|^p\omega(y)dy\right)^{\frac{1}{p}}.$$
Now, using the definition of the unweighted Morrey spaces $\mathcal{M}^{p,q}$ given in (\ref{Def_Morrey}) (\emph{i.e.} when $\omega \equiv 1$), we have
$$I_2\leq C\|f\|_{\mathcal{M}^{p,q}}\sum_{j> \lfloor \log_2(\mathcal{K})\rfloor}2^{j(\alpha-\frac{n}{q})}.$$
Since by hypothesis we have $\alpha-\frac{n}{q}<0$, the previous sum converges and we finally obtain
$$I_2\leq C \|f\|_{\mathcal{M}^{p,q}}\mathcal{K}^{(\alpha-\frac{n}{q})}.$$
\end{itemize}
With these estimates for the quantities $I_1$ and $I_2$ we have 
\begin{eqnarray*}
|\mathcal{I}_\alpha(f)(x)|&\leq &CI_1+C I_2\\
&\leq &C\mathcal{K}^\alpha\mathscr{M}(f)(x)+C \|f\|_{\mathcal{M}^{p,q}}\mathcal{K}^{(\alpha-\frac{n}{q})}.
\end{eqnarray*}
To continue we set $\mathcal{K}=\left(\frac{\|f\|_{\mathcal{M}^{p,q}}}{\mathscr{M}(f)(x)}\right)^{\frac{q}{n}}$ and we can write
$$|\mathcal{I}_\alpha(f)(x)|\leq C\mathscr{M}(f)(x)^{1-\frac{\alpha q}{n}}\|f\|_{\mathcal{M}^{p,q}}^{\frac{\alpha q}{n}}.$$
Now, taking the averaged weighted $L^{q^*}(B(x,r))$ norm of the previous inequality we obtain
\begin{eqnarray*}
\frac{1}{\omega(B(x,r))^{\frac{1}{p^*}-\frac{1}{q^*}}}\left(\int_{B(x,r)}|\mathcal{I}_\alpha(f)(y)|^{q^*}\omega(y) dy\right)^{\frac{1}{q^*}}\leq C\|f\|_{\mathcal{M}^{p,q}}^{\frac{\alpha q}{n}}\qquad \qquad\\
\times\frac{1}{\omega(B(x,r))^{\frac{1}{p^*}-\frac{1}{q^*}}}\left(\int_{B(x,r)}\left[\mathscr{M}(f)(y)^{(1-\frac{\alpha q}{n})}\right]^{q^*}\omega(y)dy\right)^{\frac{1}{q^*}},
\end{eqnarray*}
and using the definition of the weighted Morrey spaces given in (\ref{Def_Morrey}) and the properties of these spaces given in (\ref{PuissanceFracMorrey}), we obtain
$$\|\mathcal{I}_\alpha(f)\|_{\mathcal{M}^{p^*, q^*}(\omega)}\leq C\|f\|_{\mathcal{M}^{p,q}}^{\frac{\alpha q}{n}}\|\mathscr{M}(f)^{(1-\frac{\alpha q}{n})}\|_{\mathcal{M}^{p^*, q^*}(\omega)}\leq C\|f\|_{\mathcal{M}^{p,q}}^{\frac{\alpha q}{n}}\|\mathscr{M}(f)\|_{\mathcal{M}^{p^*(1-\frac{\alpha q}{n}), q^*(1-\frac{\alpha q}{n})}(\omega)}^{(1-\frac{\alpha q}{n})}.$$
Since $p^*(1-\frac{\alpha q}{n})>1$, the maximal function $\mathscr{M}$ is bounded in the Morrey space $\mathcal{M}^{p^*(1-\frac{\alpha q}{n}), q^*(1-\frac{\alpha q}{n})}(\omega)$ (recall that $\omega \in A_p$ and that $p=p^*(1-\frac{\alpha q}{n})$), so we can write
$$\|\mathcal{I}_\alpha(f)\|_{\mathcal{M}^{p^*, q^*}(\omega)}\leq C\|f\|_{\mathcal{M}^{p,q}}^{\frac{\alpha q}{n}}\|f\|_{\mathcal{M}^{p^*(1-\frac{\alpha q}{n}), q^*(1-\frac{\alpha q}{n})}(\omega)}^{(1-\frac{\alpha q}{n})}.$$
Since by hypothesis we have $p^*(1-\frac{\alpha q}{n})=p$ and $q^*(1-\frac{\alpha q}{n})=q$, we have 
$$\|\mathcal{I}_\alpha(f)\|_{\mathcal{M}^{p^*, q^*}(\omega)}\leq C\|f\|_{\mathcal{M}^{p,q}}^{\frac{\alpha q}{n}}\|f\|_{\mathcal{M}^{p,q}(\omega)}^{(1-\frac{\alpha q}{n})}.$$
To conclude, it is enough to consider the function $f=(-\Delta)^{\frac{\alpha}{2}}\phi$ and we obtain
\begin{eqnarray*}
\|\phi\|_{\mathcal{M}^{p^*, q^*}(\omega)}&\leq &C\|(-\Delta)^{\frac{\alpha}{2}}\phi\|_{\mathcal{M}^{p,q}}^{\frac{\alpha q}{n}}\|(-\Delta)^{\frac{\alpha}{2}}\phi\|_{\mathcal{M}^{p,q}(\omega)}^{(1-\frac{\alpha q}{n})}\\
&\leq &C\|\phi\|_{\dot{W}^\alpha_{p,q}}^{\frac{\alpha q}{n}}\|\phi\|_{\dot{W}^\alpha_{p,q}(\omega)}^{(1-\frac{\alpha q}{n})},
\end{eqnarray*}
which is the announced inequality since it is not difficult to deduce from it a similar estimate for a vector field $\vec{\phi}$.
\hfill $\blacksquare$
%%%%%%%%%%%%%%%%%%%%%%%%%%%%%%%%%%%%%%%%%%%%%%%%%%%
\subsection{Proof of the Lemma \ref{LemmeFullyWeighted}.}
%%%%%%%%%%%%%%%%%%%%%%%%%%%%%%%%%%%%%%%%%%%%%%%%%%%
In order to establish this lemma, we will first study a pointwise weighted inequality which is interesting on its own: 
%%%%%%%%%%%%%%%%%%%%%%%%%%%%%%%%%%%%%%%%%%%%%%%%%%%
\begin{Lemme}\label{LemmaPointwiseSubrepresentation}
Consider $0<\alpha<n$ a parameter and consider a locally integrable function $f:\mathbb{R}^n\longrightarrow \mathbb{R}$ such that we have $f\in \mathcal{M}^{p,q}(\omega)$ for $1\leq p\leq q<+\infty$ and where $\omega \in A_p$ is a Muckenhoupt weight that satisfies the ${\bf d}$-Ahlfors lower condition stated in the formula (\ref{Ahlfors}) above with $\alpha<\frac{{\bf d}}{q}$. Then we have the pointwise estimate:
$$|\mathcal{I}_\alpha(f)(x)|\leq C\mathscr{M}(f)(x)^{1-\frac{\alpha q}{{\bf d}}}\|f\|_{\mathcal{M}^{p,q}(\omega)}^{\frac{\alpha q}{{\bf d}}}.$$
\end{Lemme}
The proof of this pointwise estimate follows the same lines as the one studied above, we give some of the details for the convenience of the reader.\\ 

%%%%%%%%%%%%%%%%%%%%%%%%%%%%%%%%%%%%%%%%%%%%%%%%%%%
\noindent {\bf Proof of the Lemma \ref{LemmaPointwiseSubrepresentation}.}  Using the previous computations displayed in (\ref{EstimationAvantMorreyApoids}) we have:
\begin{eqnarray*}
|\mathcal{I}_\alpha(f)(x)|\leq c_\alpha\underbrace{\sum_{j\leq \lfloor \log_2(\mathcal{K})\rfloor}\int_{\{2^j<|y|\leq2^{j+1}\}}\frac{|f(x-y)|}{|y|^{n-\alpha}}dy}_{(I_1)}+c_\alpha\underbrace{\sum_{j> \lfloor \log_2(\mathcal{K})\rfloor}\int_{\{2^j<|y|\leq2^{j+1}\}}\frac{|f(x-y)|}{|y|^{n-\alpha}}dy}_{(I_2)},
\end{eqnarray*}
and we will study the two quantities above separately. 
%%%%%%%%%%%%%%%%%%%%%%%%%%%%%%%%%%%%%%%%%%%%%%%%%%%
\begin{itemize}
\item For the term $(1)$, by the same arguments used in (\ref{EstimationPointWise1}) above we obtain the control
$$I_1 \leq C\mathcal{K}^\alpha\mathscr{M}(f)(x).$$
%%%%%%%%%%%%%%%%%%%%%%%%%%%%%%%%%%%%%%%%%%%%%%%%%%%
\item For the term $(2)$, we will need to consider a small variant of the previous arguments. Indeed, we have:
\begin{eqnarray*}
I_2&=&\sum_{j> \lfloor \log_2(\mathcal{K})\rfloor}\int_{\{2^j<|y|\leq2^{j+1}\}}\frac{|f(x-y)|}{|y|^{n-\alpha}}dy\leq \sum_{j> \lfloor \log_2(\mathcal{K})\rfloor}2^{-j(n-\alpha)}\int_{\{2^j<|y|\leq2^{j+1}\}}|f(x-y)|dy\\
&\leq &\sum_{j> \lfloor \log_2(\mathcal{K})\rfloor}2^{-j(n-\alpha)}\int_{B(x,2^{j+1})}|f(y)|dy
\end{eqnarray*}
which we rewrite in the following manner:
$$I_2\leq \sum_{j> \lfloor \log_2(\mathcal{K})\rfloor}2^{-j(n-\alpha)}\frac{|B(x, 2^{j+1})}{|B(x, 2^{j+1})|}\int_{B(x, 2^{j+1})}|f(y)|dy=C\sum_{j> \lfloor \log_2(\mathcal{K})\rfloor}\frac{2^{j\alpha}}{|B(x, 2^{j+1})|}\int_{B(x, 2^{j+1})}|f(y)|dy.$$
Recalling that if $\omega$ is a $A_{p}$ weight with $1\leq p<+\infty$, then by using H\" older's inequality we have the following estimate
$$\frac{1}{|B|}\int_{B}|f(x)|dx\leq [\omega]_{A_{p}}^\frac{1}{{p}}\left(\frac{1}{\omega(B)}\int_B|f(x)|^{p}\omega(x)dx\right)^\frac{1}{{p}},$$
for any ball $B\subset \mathbb{R}^n$ (see the book \cite{Grafakos}, p. 285 for a proof of this fact), then we can obtain the control 
\begin{eqnarray*}
I_2&\leq &C[\omega]_{A_p}^{\frac{1}{p}}\sum_{j> \lfloor \log_2(\mathcal{K})\rfloor}2^{j\alpha}\frac{1}{\omega(B(x, 2^{j+1}))^{\frac{1}{p}}}\left(\int_{B(x, 2^{j+1})}|f(y)|^p\omega(y)dy\right)^{\frac{1}{p}}\\
&\leq &C[\omega]_{A_p}^{\frac{1}{p}}\sum_{j> \lfloor \log_2(\mathcal{K})\rfloor}2^{j\alpha}\frac{\omega(B(x, 2^{j+1}))^{-\frac{1}{q}}}{\omega(B(x, 2^{j+1}))^{\frac{1}{p}-\frac{1}{q}}}\left(\int_{B(x, 2^{j+1})}|f(y)|^p\omega(y)dy\right)^{\frac{1}{p}},
\end{eqnarray*}
for some index $p\leq q<+\infty$. At this point we use the ${\bf d}$-lower Ahlfors condition given in (\ref{Ahlfors}) with $q<{\bf d}$ to obtain the inequality $\omega(B(x, 2^{j+1}))^{-\frac{1}{q}}\leq C 2^{-(j+1)\frac{{\bf d}}{q}}$, which allows us to write
\begin{eqnarray*}
I_2&\leq &C[\omega]_{A_p}^{\frac{1}{p}}\sum_{j> \lfloor \log_2(\mathcal{K})\rfloor}2^{j\alpha}\frac{2^{-(j+1)\frac{{\bf d}}{q}}}{\omega(B(x, 2^{j+1}))^{\frac{1}{p}-\frac{1}{q}}}\left(\int_{B(x, 2^{j+1})}|f(y)|^p\omega(y)dy\right)^{\frac{1}{p}}\\
&\leq &C[\omega]_{A_p}^{\frac{1}{p}}\sum_{j> \lfloor \log_2(\mathcal{K})\rfloor}2^{j(\alpha-\frac{{\bf d}}{q})}\frac{1}{\omega(B(x, 2^{j+1}))^{\frac{1}{p}-\frac{1}{q}}}\left(\int_{B(x, 2^{j+1})}|f(y)|^p\omega(y)dy\right)^{\frac{1}{p}}.
\end{eqnarray*}
Now, using the definition of the weighted Morrey spaces $\mathcal{M}^{p,q}(\omega)$ given in (\ref{Def_Morrey}), we have
$$I_2\leq C[\omega]_{A_p}^{\frac{1}{p}} \|f\|_{\mathcal{M}^{p,q}(\omega)}\sum_{j> \lfloor \log_2(\mathcal{K})\rfloor}2^{j(\alpha-\frac{{\bf d}}{q})}.$$
Since by hypothesis we have $\alpha-\frac{{\bf d}}{q}<0$, the previous sum converges and we finally obtain
$$I_2\leq C[\omega]_{A_p}^{\frac{1}{p}} \|f\|_{\mathcal{M}^{p,q}(\omega)}\mathcal{K}^{(\alpha-\frac{{\bf d}}{q})}.$$
\end{itemize}
With these estimates for the quantities $I_1$ and $I_2$ we have 
\begin{eqnarray*}
|\mathcal{I}_\alpha(f)(x)|&\leq &CI_1+C I_2\\
&\leq &C\mathcal{K}^\alpha\mathscr{M}(f)(x)+C[\omega]_{A_p}^{\frac{1}{p}} \|f\|_{\mathcal{M}^{p,q}(\omega)}\mathcal{K}^{(\alpha-\frac{{\bf d}}{q})}.
\end{eqnarray*}
To continue we set $\mathcal{K}=\left(\frac{\|f\|_{\mathcal{M}^{p,q}(\omega)}}{\mathscr{M}(f)(x)}\right)^{\frac{q}{{\bf d}}}$ and we can write
$$|\mathcal{I}_\alpha(f)(x)|\leq C\mathscr{M}(f)(x)^{1-\frac{\alpha q}{{\bf d}}}\|f\|_{\mathcal{M}^{p,q}(\omega)}^{\frac{\alpha q}{{\bf d}}},$$
which is the wished inequality and the Lemma \ref{LemmaPointwiseSubrepresentation} is now proven. \hfill $\blacksquare$\\

%%%%%%%%%%%%%%%%%%%%%%%%%%%%%%%%%%%%%%%%%%%%%%%%%%%
With this pointwise inequality, it is not difficult to obtain the result stated in the Lemma \ref{LemmeFullyWeighted}. Indeed, setting $f=(-\Delta)^{\frac{\alpha}{2}}\phi$, we obtain
$$|\mathcal{I}_\alpha((-\Delta)^{\frac{\alpha}{2}}\phi)(x)|\leq C\mathscr{M}((-\Delta)^{\frac{\alpha}{2}}\phi)(x)^{1-\frac{\alpha q}{{\bf d}}}\|(-\Delta)^{\frac{\alpha}{2}}\phi\|_{\mathcal{M}^{p,q}(\omega)}^{\frac{\alpha q}{{\bf d}}},$$
and considering the norm $\|\cdot\|_{\mathcal{M}^{p^*, q^*}(\omega)}$, we obtain
$$\left\|\mathcal{I}_\alpha((-\Delta)^{\frac{\alpha}{2}}\phi)\right\|_{\mathcal{M}^{p^*, q^*}(\omega)}\leq C\left\|\mathscr{M}((-\Delta)^{\frac{\alpha}{2}}\phi)^{1-\frac{\alpha q}{{\bf d}}}\right\|_{\mathcal{M}^{p^*, q^*}(\omega)}\|(-\Delta)^{\frac{\alpha}{2}}\phi\|_{\mathcal{M}^{p,q}(\omega)}^{\frac{\alpha q}{{\bf d}}}.$$
Due to the properties of the weighted Morrey spaces given in (\ref{PuissanceFracMorrey}), we can write
$$\|\phi\|_{\mathcal{M}^{p^*, q^*}(\omega)}\leq C\left\|\mathscr{M}((-\Delta)^{\frac{\alpha}{2}}\phi)\right\|_{\mathcal{M}^{(1-\frac{\alpha q}{{\bf d}})p^*, (1-\frac{\alpha q}{{\bf d}})q^*}(\omega)}^{1-\frac{\alpha q}{{\bf d}}}\|(-\Delta)^{\frac{\alpha}{2}}\phi\|_{\mathcal{M}^{p,q}(\omega)}^{\frac{\alpha q}{{\bf d}}},$$
and since by hypothesis we have $1<p=(1-\frac{\alpha q}{{\bf d}})p^*\leq q=(1-\frac{\alpha q}{{\bf d}})q^*$, then the maximal function is bounded, so we obtain the control 
$$\|\phi\|_{\mathcal{M}^{p^*, q^*}(\omega)}\leq C\left\|(-\Delta)^{\frac{\alpha}{2}}\phi\right\|_{\mathcal{M}^{(1-\frac{\alpha q}{{\bf d}})p^*, (1-\frac{\alpha q}{{\bf d}})q^*}(\omega)}^{1-\frac{\alpha q}{{\bf d}}}\|(-\Delta)^{\frac{\alpha}{2}}\phi\|_{\mathcal{M}^{p,q}(\omega)}^{\frac{\alpha q}{{\bf d}}}.$$
Recalling that $(1-\frac{\alpha q}{{\bf d}})p^*=p$ and $(1-\frac{\alpha q}{{\bf d}})q^*=q$, we have
\begin{eqnarray*}
\|\phi\|_{\mathcal{M}^{p^*, q^*}(\omega)}&\leq &C\left\|(-\Delta)^{\frac{\alpha}{2}}\phi\right\|_{\mathcal{M}^{p, q}(\omega)}^{1-\frac{\alpha q}{{\bf d}}}\|(-\Delta)^{\frac{\alpha}{2}}\phi\|_{\mathcal{M}^{p,q}(\omega)}^{\frac{\alpha q}{{\bf d}}}=C\|(-\Delta)^{\frac{\alpha}{2}}\phi\|_{\mathcal{M}^{p,q}(\omega)}\\
&\leq & C\|\phi\|_{\dot{W}^{\alpha}_{p,q}(\omega)},
\end{eqnarray*}
which is the announced estimate and the Lemma \ref{LemmeFullyWeighted} is now completely proven, as the same inequality can be established for a vector field $\vec{\phi}$. \hfill $\blacksquare$
%%%%%%%%%%%%%%%%%%%%%%%%%%%%%%%%%%%%%%%%%%%%%%%%%%%
\section{A weighted uniqueness result}\label{AppendixUniqueness}
In this section we explore a weighted version of the Theorem \ref{Theorem_Uniqueness}. Namely, we deduce the uniqueness of a solution $\vu\in \dot{W}^1_{\frac{p}{2},\frac{n}{2}}\cap \dot{W}^1_{\frac{p}{2},\frac{n}{2}}(\omega)$ under a weighted hypothesis only and we will assume that the weight $\omega$ satisfies a ${\bf n}$-lower Ahlfors condition.\\

Let  $\vu \in \dot{W}^{1}_{\frac{p}{2},\frac{n}{2}}\cap \dot{W}^{1}_{\frac{p}{2},\frac{n}{2}}(\omega)$ and $\vv \in \dot{W}^{1}_{\frac{p}{2},\frac{n}{2}}\cap \dot{W}^{1}_{\frac{p}{2},\frac{n}{2}}(\omega)$ be two solutions of the stationary Navier-Stokes equations \eqref{Equation_Intro} and let us assume that $\vu\neq \vv$. Considering the expression (\ref{FormuleBilineaireUnicite}) and applying the norm $\dot{W}^{1}_{\frac{p}{2}, \frac{n}{2}}(\omega)$ to the both sides of this expression we have 
\begin{eqnarray*}
\|\vu-\vv\|_{\dot{W}^{1}_{\frac{p}{2}, \frac{n}{2}}(\omega)}&\leq &\|(-\Delta)^{\frac{1}{2}}(-\Delta)^{-1}\mathbb{P}(div(\mathcal{T}(\vu-\vv)\otimes \vu))\|_{\mathcal{M}^{\frac{p}{2}, \frac{n}{2}}(\omega)}\\
&&+\|(-\Delta)^{\frac{1}{2}}(-\Delta)^{-1}\mathbb{P}(div(\mathcal{T}(\vv)\otimes (\vv-\vu)))\|_{\mathcal{M}^{\frac{p}{2}, \frac{n}{2}}(\omega)},
\end{eqnarray*}
and by the properties of the Leray projector and the boundedness of the Riesz transforms in Morrey spaces, we get $\|\vu-\vv\|_{\dot{W}^{1}_{\frac{p}{2}, \frac{n}{2}}(\omega)}\leq \|\mathcal{T}(\vu-\vv)\otimes \vu\|_{\mathcal{M}^{\frac{p}{2}, \frac{n}{2}}(\omega)}+\|\mathcal{T}(\vv)\otimes (\vv-\vu)\|_{\mathcal{M}^{\frac{p}{2}, \frac{n}{2}}(\omega)}$. By the H\"older inequality in Morrey spaces, we obtain
$$\|\vu-\vv\|_{\dot{W}^{1}_{\frac{p}{2}, \frac{n}{2}}(\omega)}\leq \|\mathcal{T}(\vu-\vv)\|_{\mathcal{M}^{p,n}(\omega)}\|\vu\|_{\mathcal{M}^{p,n}(\omega)}+\|\mathcal{T}(\vv)\|_{\mathcal{M}^{p,n}(\omega)} \|\vv-\vu\|_{\|_{\mathcal{M}^{p,n}}(\omega)}.$$
Applying the inequality $\|\mathcal{T}(\vec{\phi})\|_{\mathcal{M}^{p,n}(\omega)}\leq C\|\vec{\phi}\|_{\dot{W}^{1}_{\frac{p}{2}, \frac{n}{2}}(\omega)}$ (see the third point of the Lemma \ref{LemmaRough}, recall that we are assuming a ${\bf n}$-lower Ahlfors condition) and the estimate $\|\vec{\phi}\|_{\mathcal{M}^{p,n}(\omega)}\leq C\|\vec{\phi}\|_{\dot{W}^{1}_{\frac{p}{2}, \frac{n}{2}}(\omega)}$ (see Lemma \ref{LemmeFullyWeighted}), we obtain the control
\begin{eqnarray*}
\|\vu-\vv\|_{\dot{W}^{1}_{\frac{p}{2}, \frac{n}{2}}(\omega)}&\leq &C\|\vu-\vv\|_{\dot{W}^{1}_{\frac{p}{2}, \frac{n}{2}}(\omega)}\|\vu\|_{\dot{W}^{1}_{\frac{p}{2}, \frac{n}{2}}(\omega)}+C\|\vv\|_{\dot{W}^{1}_{\frac{p}{2}, \frac{n}{2}}(\omega)} \|\vv-\vu\|_{\dot{W}^{1}_{\frac{p}{2}, \frac{n}{2}}(\omega)}\\
&\leq & C(\|\vu\|_{\dot{W}^{1}_{\frac{p}{2}, \frac{n}{2}}(\omega)}+\|\vv\|_{\dot{W}^{1}_{\frac{p}{2}, \frac{n}{2}}(\omega)}) \|\vv-\vu\|_{\dot{W}^{1}_{\frac{p}{2}, \frac{n}{2}}(\omega)}. 
\end{eqnarray*}
Since we assumed that $\vu\neq \vv$, we have $\|\vu-\vv\|_{\dot{W}^{1}_{\frac{p}{2}, \frac{n}{2}}(\omega)}=\|\vv-\vu\|_{\dot{W}^{1}_{\frac{p}{2}, \frac{n}{2}}(\omega)}\neq 0$, and we obtain the inequality $1\leq  C(\|\vu\|_{\dot{W}^{1}_{\frac{p}{2}, \frac{n}{2}}(\omega)}+\|\vv\|_{\dot{W}^{1}_{\frac{p}{2}, \frac{n}{2}}(\omega)})$. Therefore if the quantity $\|\vu\|_{\dot{W}^{1}_{\frac{p}{2}, \frac{n}{2}}(\omega)}+\|\vv\|_{\dot{W}^{1}_{\frac{p}{2}, \frac{n}{2}}(\omega)}$ is small enough, then we obtain a contradiction, from which we deduce that 
$0=\|\vu-\vv\|_{\dot{W}^{1}_{\frac{p}{2}, \frac{n}{2}}(\omega)}$, since by the Sobolev embedding given in the Lemma \ref{LemmaRough} this quantity controls the term $\|\vu-\vv\|_{\mathcal{M}^{p,n}(\omega)}$ then we must have $\vu=\vv$. We thus have the following result: 
%%%%%%%%%%%%%%%%%%%%%%%%%%%%%%%%%%%%%%%%%%%%%%%%%%%
\begin{Lemme}
Consider $\vu, \vv\in \dot{W}^1_{\frac{p}{2},\frac{n}{2}}\cap \dot{W}^1_{\frac{p}{2},\frac{n}{2}}(\omega)$ be two solutions associated to the same external force $\vf\in \dot{W}^{-1}_{\frac{p}{2}, \frac{n}{2}}\cap \dot{W}^{-1}_{\frac{p}{2}, \frac{n}{2}}(\omega)$ of the equation (\ref{Equation_Intro}) where $\omega\in A_{\frac{p}{2}}$ with $2<p< n$ is a Muckenhoupt weight that satisfies the ${\bf n}$-lower condition (\ref{Ahlfors}). If $\|\vu\|_{\dot{W}^{1}_{\frac{p}{2}, \frac{n}{2}}(\omega)}+\|\vv\|_{\dot{W}^{1}_{\frac{p}{2}, \frac{n}{2}}(\omega)}$ is small enough, then we have $\vu=\vv$.
\end{Lemme}
%%%%%%%%%%%%%%%%%%%%%%%%%%%%%%%%%%%%%%%%%%%%%%%%%%% 
\noindent This result is the weighted counterpart of the Theorem \ref{Theorem_Uniqueness} and by separating the information we can also deduce the uniqueness of solutions (related to the same external force) outside the scope of the Theorem \ref{Theorem_Existence}. Indeed, if $\vu \in \dot{W}^1_{\frac{p}{2},\frac{n}{2}}\cap \dot{W}^1_{\frac{p}{2},\frac{n}{2}}(\omega)$ is a (small) solution obtained via the Theorem \ref{Theorem_Existence} and if $\vv$ is also a solution of the equation (\ref{Equation_Intro}) with $\vv\in \dot{W}^1_{\frac{p}{2},\frac{n}{2}}\cap \dot{W}^1_{\frac{p}{2},\frac{n}{2}}(\omega)$ but where $\|\vv\|_{\dot{W}^1_{\frac{p}{2},\frac{n}{2}}}\gg1$ and $\|\vv\|_{\dot{W}^1_{\frac{p}{2},\frac{n}{2}}(\omega)}\ll \epsilon$ (note that since the solution $\vv$ is not small, it can not be obtained via the Theorem \ref{Theorem_Existence}), then by the previous lemma we can still deduce that $\vu\equiv \vv$.\\

To conclude, let us remark that a simple adaptation of the computations performed in the previous lines allows us to deduce the following uniqueness result in the setting of the fully weighted existence result presented in the Theorem \ref{Theo_FullyWeighted} above: 
%%%%%%%%%%%%%%%%%%%%%%%%%%%%%%%%%%%%%%%%% %%%%%%%%%
\begin{Lemme}
Consider $\vu, \vv\in\dot{W}^1_{\frac{p}{2},\frac{\bf d}{2}}(\omega)$ be two solutions associated to the same external force $\vf\in \dot{W}^{-1}_{\frac{p}{2}, \frac{\bf d}{2}}(\omega)$ of the equation (\ref{Equation_Intro}) where $\omega\in A_{\frac{p}{2}}$ with $2<p< n$ is a Muckenhoupt weight that satisfies the ${\bf d}$-lower condition (\ref{Ahlfors}). If $\|\vu\|_{\dot{W}^{1}_{\frac{p}{2}, \frac{\bf d}{2}}(\omega)}+\|\vv\|_{\dot{W}^{1}_{\frac{p}{2}, \frac{\bf d}{2}}(\omega)}$ is small enough, then we have $\vu=\vv$.
\end{Lemme}
%%%%%%%%%%%%%%%%%%%%%%%%%%%%%%%%%%%%%%%%% %%%%%%%%%
\noindent {\bf Acknowledgment.} This work was supported by the GDRI ECO-Math.\\

\noindent {\bf Conflict of interest.} We declare that we do not have any commercial or associative interest that represents a conflict of interest in connection with the work submitted.

%%%%%%%%%%%%%%%%%%%%%%%%%%%%%%%%%%%%%%%%%%%%%%%%%%%%%%%%%%%%%%%%%%%%%%%%%%%%%%%%%%%%%%%%%%%%%%%%%%%%%%%%%%%%%%%%%%%%%%%%%%%%%%%%%%%%%%%%%%%%%%%%%%%%%%%%%%%%%%%%%%%%%%%%%%%%%%%%%%%%%%%%%%%%%%%%%%%%%%%%%%%%
\end{document}